\documentclass[letterpaper,11pt]{amsart}
\usepackage[margin=1.2in]{geometry}
\usepackage{amsmath,amsthm,amssymb}
\usepackage{xspace,xcolor}
\usepackage[breaklinks,colorlinks,citecolor=teal,linkcolor=teal,urlcolor=teal,pagebackref,hyperindex]{hyperref}
\usepackage[alphabetic]{amsrefs}
\usepackage[all]{xy}
\usepackage[english]{babel}
\usepackage{enumitem}
\usepackage{tikz,xcolor}
\usepackage{tikz-cd}
\usepackage{mathrsfs}
\usepackage{color}

\newcommand{\bV}{\mathbf V}
\newcommand{\bS}{\mathbf S}

\newcommand{\bhW}{\widehat{\mathbf W}}
\newcommand{\bhV}{\widehat{\mathbf V}}

\theoremstyle{plain}
\newtheorem{thm}{Theorem}[section]

\newtheorem{lem}[thm]{Lemma}
\newtheorem{prop}[thm]{Proposition}
\newtheorem{cor}[thm]{Corollary}

\theoremstyle{definition}

\newtheorem{eg}[thm]{Example}

\theoremstyle{remark}
\newtheorem{rmk}[thm]{Remark}

\def\Mustata{Mus\-ta\-\c{t}\u{a}\xspace}

\def\N{{\mathbf N}}
\def\Z{{\mathbf Z}}
\def\Q{{\mathbf Q}}

\def\C{{\mathbf C}}

\def\A{{\mathbf A}}
\def\P{{\mathbf P}}

\def\cD{\mathcal{D}}
\def\cE{\mathcal{E}}

\def\cG{\mathcal{G}}
\def\cH{\mathcal{H}}

\def\cK{\mathcal{K}}

\def\cM{\mathcal{M}}
\def\cN{\mathcal{N}}
\def\cO{\mathcal{O}}

\def\cU{\mathcal{U}}

\def\bS{{\bf S}}

\def\bN{{\bf N}}

\def\.{\cdot}
\def\^{\widehat}

\def\de{\partial}

\def\({\left(}
\def\){\right)}

\renewcommand{\and}{ \ \ \text{ and } \ \ }

\begin{document}

\title[Fourier Transform and Radon Transform for MHM]{Fourier Transform and Radon Transform for Mixed Hodge modules}
\author[B.~Dirks]{Bradley Dirks}

\address{Department of Mathematics, Stony Brook University, Stony Brook, NY 11794-3651, USA}

\email{bradley.dirks@stonybrook.edu}

\thanks{The author was supported by NSF-MSPRF grant DMS-2303070.}


\maketitle

\begin{abstract} We give a generalization to bi-filtered $\cD$-modules underlying mixed Hodge modules of the relation between microlocalization along $f_1,\dots, f_r \in \cO_X(X)$ and vanishing cycles along $g = \sum_{i=1}^r y_i f_i$. This leads to an interesting relation between localization triangles.

As an application, we use these results to compare the $k$-plane Radon transform and the Fourier-Laplace transform for mixed Hodge modules. This is then applied to the Hodge module structure of certain GKZ systems.
\end{abstract}

\section{Introduction} In \cite{CD}, the Fourier-Laplace transform for monodromic mixed Hodge modules is studied. The transform is related to various functors from geometry, using as motivation the corresponding results for perverse sheaves \cite{KashShap}*{(10.3.31)}. Then, because these functors preserve the property of being a mixed Hodge module, this endows the Fourier-Laplace transform with a mixed Hodge module structure. As noted in \cite{MonoMHM1}*{}, the Fourier-Laplace transform of a monodromic $\cD$-module can actually underlie several distinct mixed Hodge modules, so throughout the present paper it is important to consistently use the structure as defined in \cite{CD}, which we briefly review here, to establish notation.

Let $M$ be a monodromic mixed Hodge module on the trivial vector bundle $E = X\times \A^r_z$, where $z_1,\dots, z_r$ are the coordinates on $\A^r_z$. Let $E^\vee = X\times \A^r_y$ be the dual bundle, with coordinates $y_1,\dots, y_r$. Denote the product as $\cE = E \times_X E^\vee = X\times \A^r_z \times \A^r_y$ with $g = \sum_{i=1}^r y_iz_i  \in \cO(\cE)$. Let $\sigma\colon E^\vee \to \cE$ be the inclusion of the zero section and let $p\colon \cE \to E, q\colon\cE \to E^\vee$ be the natural projections. The author and Chen constructed in \cite{CD} an isomorphism of $\cD_{E^\vee}$-modules
\[ {\rm FL}(\cM) = \cH^0\sigma^* \phi_gp^!(\cM)[-r].\]

In Section \ref{sect-FL} we make some remarks concerning the functor ${\rm FL}$ which are useful in the applications. There are a few important results: Proposition \ref{prop-main} shows that the Fourier-Laplace transform is the only non-vanishing cohomology of $\sigma^* \phi_g p^!(\cM)[-r]$. Moreover, the functor $\sigma^*$ could be replaced with $\sigma^!$ (see Corollary \ref{cor-lastFunctor}). Corollary \ref{cor-ProperFL} applies Proposition \ref{prop-main} to see that the Fourier-Laplace transform commutes with proper push-forward of the base, i.e., proper maps $f\colon X\to Y$. These results are known for constructible complexes and the Fourier-Sato transform, but our proofs are rather simple using Lemma \ref{lem-supportVanCycles}.

We then give a generalization of \cite{D-Microlocal}*{Prop. 3.4} to bi-filtered $\cD$-modules underlying mixed Hodge modules. Let $i_Z \colon Z \to X$ be the embedding of a closed subvariety defined by $f_1,\dots ,f_r \in \cO_X(X)$. Let $j_Z \colon X \setminus Z \to X$ be the inclusion of the complement.

Given a mixed Hodge module $M$ on $X$, one gets a mixed Hodge module $\Gamma_* M$ on $X\times \A^r_t$ by pushing forward along the graph embedding $\Gamma\colon X\to X\times \A^r_t$. By applying the Verdier specialization functor, we get a monodromic mixed Hodge module ${\rm Sp}_Z(M) = {\rm Sp}(\Gamma_* M)$ on $X\times \A^r_z$. Finally, the Fourier-Laplace transform yields a mixed Hodge module $\mu_Z(M) = \mu(\Gamma_* M) = {\rm FL}{\rm Sp}(\Gamma_* M)$ on $X\times \A^r_y$, the microlocalization of $M$ along $Z$.

Alternatively, one can start with $M$ and consider $\pi^!(M)[-r]$, a mixed Hodge module on $X\times \A^r_y$, where $\pi \colon X\times \A^r_y \to X$ is the projection. There is a globally defined function $g = \sum_{i=1}^r y_i f_i$ on $X\times \A^r_y$, hence we get a mixed Hodge module $\phi_g(\pi^!(M)[-r])$ on $X\times \A^r_y$. Let $(\cM,F,W)$ be the bi-filtered $\cD_X$-module underlying $M$.

\begin{thm} \label{thm-Microlocalization} There is a natural isomorphism of bi-filtered $\cD_{X\times \A^r_y}$-modules
\[ \phi_g(\pi^!(\cM)[-r],F,W) \cong (\mu_Z(\cM),F,W).\]
\end{thm}

This theorem differs from that in loc. cit. because, in this paper, we use the standard conventions for indexing the Hodge filtration of filtered left $\cD$-modules and we use $\pi^!(-)[-r]$ instead of $(-) \boxtimes \Q^H_{\A^r_y}[r]$, which differs by a Tate twist.

\begin{rmk} The hypersurface $g = \sum_{i=1}^r y_i f_i$ is used in geometric representation theory for a technique called ``dimensional reduction". See \cite{Davison} and the references therein.

In fact, the isomorphism of Theorem \ref{thm-Microlocalization} is studied in \cite{Schefers} for perverse sheaves. We do not know how to compare the isomorphism we obtain with the one that they have, though we suspect they give the same isomorphism. This would then enhance the isomorphism of Theorem \ref{thm-Microlocalization} to an isomorphism of $\Q$-mixed Hodge modules. 

It would be very nice to find a ``motivic" incarnation of this isomorphism, though we cannot see how that might be done at the moment.
\end{rmk}

In the application to Radon transforms, we will need a related result which holds at the level of mixed Hodge modules. The idea is to apply the construction of Theorem \ref{thm-Microlocalization} to the localization exact triangle
\[ \sigma_* \sigma^! \pi^! \to \pi^!\to j_* j^* \pi^! \xrightarrow[]{+1},\]
where $\sigma\colon X \to X\times \A^r_y$ is the inclusion of the zero section and $j$ is the inclusion of the complement of the zero section. By applying $\sigma^! \phi_{g,1}$ to this triangle, we get a triangle
\[ {\rm id}\to \sigma^! \phi_{g,1} \pi^! \to \sigma^! \phi_{g,1} \pi^! \xrightarrow[]{+1}.\]

Motivated by Theorem \ref{thm-Microlocalization}, the central term should be isomorphic, up to a shift, to $\sigma^! {\rm FL}{\rm Sp}_Z(-)$. By Remark \ref{rmk-FLRestriction} below, this is isomorphic to $i^*{\rm Sp}_Z(-)[-r]$, where $i\colon X \to X\times \A^r_z$ is the inclusion of the zero section. By \cite{SaitoMHM}*{Pg. 269}, this is naturally isomorphic to $i_{Z*} i_Z^*(-)[-r]$. With a little argument, this line of thought leads to the following:

\begin{thm} \label{thm-MHMLocalization} There is a natural isomorphism of triangles
\[ \sigma^! \phi_{g,1} \sigma_* \sigma^! \pi^! \to \sigma^! \phi_{g,1} \pi^!\to \sigma^! \phi_{g,1} j_* j^* \pi^!\xrightarrow[]{+1}\]
\[ {\rm id} \to i_{Z*}i_Z^* \to j_{Z!} j_Z^{*}[1] \xrightarrow[]{+1}.\]

Dually, there is a natural isomorphism of triangles
\[ \sigma^* \phi_{g,1} j_! j^* \pi^* \to \sigma^* \phi_{g,1} \pi^* \to \sigma^* \phi_{g,1} \sigma_* \sigma^* \pi^* \xrightarrow[]{+1}\]
\[ j_{Z*}j_Z^*[-1] \to i_{Z*} i_Z^!\to {\rm id}\xrightarrow[]{+1}.\]
\end{thm}

\begin{rmk} In view of Remark \ref{rmk-MonodromicPushforward}, the result above is related to \cite{Davison}*{Thm. A.1}. Davison shows that for constructible complexes, there is a natural isomorphism 
\[ \pi_! \phi_{g,1} \pi ^! \cong \pi_! \pi^! i_{Z*}i_Z^*,\]
where the right hand side is isomorphic to $i_{Z*} i_Z^*$.

In fact, our proof of Theorem \ref{thm-MHMLocalization} uses this result, noting that by Remark \ref{rmk-MonodromicPushforward} we can replace $\sigma^!$ with $\pi_!$ and $\sigma^*$ with $\pi_*$.
\end{rmk}

The final main result concerns the Radon transform for $\cD$-modules, whose definition (originally due to Brylinski \cite{Brylinski}) we recall here, mostly following the notation of \cite{ReicheltRadon}. Let $\bV$ be an $(n+1)$-dimensional affine space with $\bV^0$ the complement of the origin. Let $\P = \P(\bV)$ be the projective space of $\bV$, with the canonical projection $\pi\colon \bV^0 \to \P$. Consider the dual space $\bhV$ and for $k \in \{1,\dots ,n\}$, let $\bhW = \bhV^k$ be the $k$-fold product of the dual space. Points of this space give subspaces of $\bV$ of codimension less than or equal to $k$. Let $r\colon \bS(k,n) \hookrightarrow \bhW$ be the open subset corresponding to subspaces of codimension equal to $k$.

The \emph{Grassmannian} of $(n-k)$-planes in $\P$ is a quotient of $\bS(k,n)$ by a natural ${\rm GL}_k(\C)$-action. We denote it by $\mathbf G$ with projection $\pi\colon \bS(k,n) \to \mathbf G$. We consider the incidence variety $i_Z\colon Z \hookrightarrow \P \times \mathbf G$ defined by
\[ \{ (\pi(v), [\lambda_1,\dots, \lambda_k]) \mid \lambda_1(v) = \dots = \lambda_k(v) = 0\},\]
with complement $j_C\colon C \to \P \times \mathbf G$. We have the diagrams
\[ \begin{tikzcd} & C \ar[dr,"\pi_1^C"] \ar[dl, swap,"\pi_2^C"] \ar[d] & \\ \P & \P \times \mathbf G \ar[r,"\pi_2"] \ar[l,swap,"\pi_1"] & \mathbf G \\ & Z \ar[ur,swap,"\pi_1^Z"] \ar[ul,"\pi_2^Z"] \ar[u] & \end{tikzcd}, \quad  \begin{tikzcd} & C' \ar[dr,"\pi_1^{C'}"] \ar[dl, swap,"\pi_2^{C'}"] \ar[d] & \\ \P & \P \times \bhW \ar[r,"\rho_W"] \ar[l,swap,"\overline{p}"] & \bhW \\ & Z' \ar[ur,swap,"\pi_1^{Z'}"] \ar[ul,"\pi_2^{Z'}"] \ar[u] & \end{tikzcd}\]
where $Z'$ is similarly an incidence variety. We can define various Radon transform triangles from these diagrams, see the statement of Theorem \ref{thm-Radon} below.

The ``affine'' Radon transforms (those coming from the diagram with $\bhW$) can be rewritten as Fourier-Laplace transforms via the following construction: consider $\bV^0 \times_{\P} \bV^0 \times_{\P} \dots \times_{\P} \bV^0$, which is the $k$-fold fiber product of $\bV^0$ over $\P$. An element of this space is a tuple $(v_1,\dots, v_k)$ of non-zero vectors which all lie on the same line in $\bV$. This has a morphism
\[ \bV^0 \times_{\P} \dots \times_{\P} \bV^0 \to \P \times \bV^k, \quad (v_1,\dots, v_k) \mapsto (\pi(v_1), v_1,\dots, v_k),\]
and we denote by $\widetilde{\bV}$ the closure of the image. As a variety over $\P$, one can identify $\widetilde{\bV}$ with the total space of the vector bundle $\bigoplus_{i=1}^k \cO_{\P}(-1)$. We have the diagram
\[ \begin{tikzcd} \widetilde{\bV} \ar[d,"\widetilde \pi"] \ar[r,"\widetilde{\iota}"] \ar[rd,"\widetilde{\jmath}"] & \P \times \bV^k \ar[dl,dotted] \ar[d,"\rho_V"] \\ \P & \bV^k\end{tikzcd},\]
where we denote the dotted arrow by $\overline{q}$. We let $E = \P \times \{0\} \subseteq \widetilde{\bV}$ be the zero section and $\widetilde{\bV}^\circ$ its complement. We have the diagram
\[ \begin{tikzcd} & E \ar[d] \ar[dr,"\eta"] \ar[dl,swap,"{\rm id}"] & \\ \P & \widetilde{\bV} \ar[r,"\widetilde{\jmath}"] \ar[l,swap,"\widetilde{\pi}"] & \bV^k \\ & \widetilde{\bV}^0 \ar[ru,swap,"\widetilde{\jmath}_\circ"] \ar[lu,"\widetilde{\pi}_\circ"] \ar[u] & \end{tikzcd},\]
where $\eta\colon E \to \bV^k$ is the projection onto the origin.

We can then pull objects on $\P$ to $\widetilde{\bV}$, push them forward to $\bV^k$, and then apply the Fourier-Laplace transform to get objects on $\bhW$. The main theorem is the following analogue of \cite{RadonFourier}*{Thm. 2} in the case $k=1$ and \cite{ReicheltGKZ}*{Sect. 2.3} in the case $k > 1$. 

\begin{thm} \label{thm-Radon} In the notation above, there is a natural isomorphism between exact triangles of functors from $D^b({\rm MHM}(\P))$ to $D^b({\rm MHM}(\mathbf S(k,n)))$:
\[ r^*({\rm FL} \eta_* \to {\rm FL} \widetilde{\jmath}_* \widetilde{\pi}^! \to {\rm FL}\widetilde{\jmath}_{\circ *} \widetilde{\pi}^!_\circ \xrightarrow[]{+1})\]
\[ \pi^*(\pi_{2*} \pi_1^* \to \pi_{2*}^Z \pi_1^{Z*} \to \pi_{2!}^{C} \pi_1^{C*}[1]  \xrightarrow[]{+1} )(k(n+1))[k(n+1)].\]

Dually, there is a natural isomorphism between exact triangles:
\[ r^*({\rm FL}\widetilde{\jmath}_{\circ !} \widetilde{\pi}^*_\circ \to {\rm FL} \widetilde{\jmath}_* \widetilde{\pi}^* \to {\rm FL} \eta_*  \xrightarrow[]{+1})\]
\[ \pi^!( \pi_{2*}^C \pi_{1}^{C!}[-1] \to \pi_{2!}^Z \pi_1^{Z!} \to \pi_{2!} \pi_1^! \xrightarrow[]{+1})[-k(n+1)].\]
\end{thm}

In fact, this follows from a version which directly compares the affine Radon transform with the Fourier-Laplace transform. See Proposition \ref{prop-affineRadon} below for the precise statement.

\begin{rmk} In the case $k =1$, we have $\mathbf G = \P^*$, the dual projective space. In this case, $\widetilde{\bV}$ can be identified with the blow-up of $\bV$ at the origin.
\end{rmk}

\noindent {\bf Outline of Paper.} In Section \ref{sect-background}, we mention briefly the important aspects of the theory of mixed Hodge modules and $V$-filtrations which are needed for the paper. We also recall the definition and properties of the Fourier-Laplace transform of \cite{CD}.

In Section \ref{sect-FL}, we prove Proposition \ref{prop-main} and Lemma \ref{lem-supportVanCycles}. The important input is an approximation of the support of vanishing cycles from \cite{MaximVanCycles}.

In Section \ref{sect-Microlocal}, we prove Theorem \ref{thm-Microlocalization}. The point is that all the arguments from \cite{CDMO}*{Sect. 3} and \cite{D-Microlocal}*{Sect. 3} do not require the underlying $\cD_X$-module to be equal to $\cO_X$. We also keep track of the weight filtration. We end with the localization triangle version, Theorem \ref{thm-MHMLocalization}, which follows quickly from \cite{Davison}*{Thm. A.1}.

In the penultimate Section \ref{sect-Radon}, we prove Theorem \ref{thm-Radon}. The point is to follow the arguments of \cite{RadonFourier}, but we do not work with the convolution as they do (because we do not know how to make sense of their integral kernels in the setting of mixed Hodge modules). Interestingly, Theorem \ref{thm-MHMLocalization} provides precisely the missing piece which is needed for the comparison theorem.

The last section applies the comparison of the hyperplane Radon and Fourier-Laplace transforms (as in the version \ref{prop-affineRadon}) to a comparison between two mixed Hodge module structures on certain GKZ systems. We define the GKZ systems of interest and then prove the comparison, which is an easy consequence of the proposition.

\noindent {\bf Acknowledgements}. The author would like to thank James Hotchkiss for discussions about the Fourier-Laplace transform. He would also like to thank Thomas Reichelt and Uli Walther for asking about the connection with the Radon transform. The author is grateful to Qianyu Chen, Ben Davison, Martin Gallauer, Radu Laza, Lauren\c{t}iu Maxim, Mircea \Mustata, Sebasti\'{a}n Olano, Claude Sabbah and Christian Schnell for useful conversations about ideas appearing in this paper.

\section{Background} \label{sect-background} We set the notation and review basic facts in this section. We will use the theory of $\cD$-modules without a review. For details, one should consult \cite{HTT}.
\subsection{Mixed Hodge modules} On any reduced complex algebraic variety $X$, Morihiko Saito defined \cite{SaitoMHP,SaitoMHM} the category of mixed Hodge modules ${\rm MHM}(X)$, which is essentially an extension of the category of generically defined admissible variations of mixed Hodge structures on $X$. For details, one should read Saito's papers or Sabbah-Schnell's mixed Hodge modules project \cite{MHMProj}, and for a nice overview of the theory, one should consult Schnell's survey \cite{Schnell}.

When $X$ is smooth, the data of a Hodge module $M$ on $X$ consists of a bi-filtered regular holonomic $\cD_X$-module $(\cM,F,W)$ where $F_\bullet \cM$ is an increasing good filtration by sub-$\cO_X$-modules (called the \emph{Hodge} filtration) and $W_\bullet \cM$ is a finite increasing filtration by sub-$\cD_X$-modules (called the \emph{weight} filtration), along with a filtered $\Q$-perverse sheaf $(\cK,W)$ and a filtered quasi-isomorphism
\[ \alpha\colon (\cK,W)\otimes_{\Q} \C \cong {\rm DR}_X(\cM,W).\]

Such data are subject to a plethora of conditions, defined using the $V$-filtration along hypersurfaces (the basics of which we review in the next subsection).

\begin{rmk} The \emph{trivial Hodge module} on a smooth variety $X$ is denoted $\Q_X^H[\dim X]$, its underlying bi-filtered $\cD_X$-module is $(\cO_X,F,W)$ where ${\rm gr}^F_0 \cO_X = \cO_X$ and ${\rm gr}^W_{\dim X}\cO_X = \cO_X$.
\end{rmk}

A morphism $\Phi\colon M \to N$ of mixed Hodge modules consists of $\Phi_{\cD}\colon (\cM,F,W) \to (\cN,F,W)$ and $\Phi_{\Q}\colon (\cK_M,W) \to (\cK_N,W)$ which are compatible (after complexification) under the Riemann-Hilbert correspondence. The bi-filtered morphism $\Phi_{\cD}$ is automatically bi-strict with respect to $(F_\bullet,W_\bullet)$, hence, the category ${\rm MHM}(X)$ is abelian. 

As in \cite{SaitoMHM}*{(2.17.7)}, there are \emph{Tate twist} functors $M \mapsto M(\ell)$, where the underlying regular holonomic $\cD_X$-module $\cM$ is unchanged, but the Hodge and weight filtrations are shifted:
\[ M(\ell) = (\cM,F_{\bullet - \ell},W_{\bullet + 2\ell}).\]

There is a duality functor:
\[ \mathbf D_X\colon D^b({\rm MHM}(X)) \to D^b({\rm MHM}(X))^{\rm op}\] and given $f\colon X \to Y$ there are functors
\[ f_*,f_!\colon D^b({\rm MHM}(X)) \to D^b({\rm MHM}(Y)),\]
\[ f^!,f^*\colon D^b({\rm MHM}(Y)) \to D^b({\rm MHM}(X)),\]
and which are related in the following way: the functor $f_!$ is left adjoint to $f^!$, the functor $f^*$ is left adjoint to $f_*$, and there are natural isomorphisms
\[ f_! \cong \mathbf D_Y \circ f_* \circ \mathbf D_X,\]
\[ f^* \cong \mathbf D_X \circ f^! \circ \mathbf D_Y.\]

These behave well with composition, in the following way: let $X \xrightarrow[]{f} Y \xrightarrow[]{g} Z$ be two maps between varieties. Then
\[ g_! f_! =(g\circ f)_!, \quad g_* f_* = (g\circ f)_*,\quad  f^! g^! = (g\circ f)^!, \text{ and } f^* g^* = (g\circ f)^*.\]

On underlying $\cD$-modules, these functors agree with the functors defined for complexes with regular holonomic cohomology (see \cite{HTT}*{Chapter 7} for a review of these functors in that case).

\begin{eg} When $X$ is smooth, the trivial Hodge module $\Q_X^H[\dim X]$ is almost self dual: we have an isomorphism (called a \emph{polarization})
\[ \mathbf D_X(\Q_X^H[\dim X]) \cong \Q_X^H[\dim X](\dim X).\]

When $f\colon X \to Y$ is smooth and surjective, the functor $f^*$ is almost self dual. Namely, we have a natural isomorphism
\[ f^!(-) \cong f^*((-)(c))[2c],\]
where $c = \dim X - \dim Y$.

For any $f\colon X\to Y$, there is a natural transformation $f_! \to f_*$. When $f$ is proper, this map is a natural isomorphism.
\end{eg}

\begin{eg}[Saito's Base Change Result] \label{eg-BaseChange} Saito proved a general base change result for these functors \cite{SaitoMHM}*{(4.4.3)}. Namely, if
\[ \begin{tikzcd} W \ar[r,"g' "] \ar[d,"f' "] & X \ar[d,"f"] \\ Z \ar[r,"g"] & Y\end{tikzcd}\]
is a Cartesian diagram, then there are natural isomorphisms
\[ g^! f_* \cong (f')_* (g')^!, \quad g^* f_! \cong (f')_! (g')^*.\]

When $g$ (hence $g'$) is smooth and surjective, then we get Smooth Base Change:
\[g^! f_! \cong (f')_! (g')^! \text{ and } g^* f_* \cong (f')_* (g')^*.\]

When $f$ (hence $f'$) is proper, then we get Proper Base Change:
\[g^! f_! \cong (f')_! (g')^!\text{ and }g^* f_* \cong (f')_* (g')^*.\]
\end{eg}

\subsection{$V$-filtrations} Let $M$ be a mixed Hodge module on $X$ and let $f_1,\dots, f_r \in \cO_X(X)$ be globally defined functions on $X$. If $\Gamma\colon X\to X\times \A^r_t$ is the graph embedding along $f_1,\dots, f_r$, then $\Gamma_*(M)$ is a mixed Hodge module on $X\times \A^r_t$. We can consider the underlying $\cD_{X\times \A^r_t}$-module $\Gamma_+(\cM) = \bigoplus_{\alpha \in \N^r} \cM\de_t^\alpha \delta_f$.

The action by differential operators is given, in local coordinates $x_1,\dots, x_n$ on $X$, by
\[ h(m \de_t^\alpha \delta_f) = hm \de_t^\alpha \delta_f \text{ for all } h \in \cO_X,\]
\[ \de_{x_i}(m \de_t^\alpha \delta_f) = \de_{x_i}(m) \de_t^\alpha \delta_f - \sum_{j=1}^r \de_{x_i}(f_j) m \de_t^{\alpha+e_j} \delta_f,\]
\[ t_i (m \de_t^\alpha \delta_f) = f_i m \de_t^\alpha \delta_f - \alpha_i m \de_t^{\alpha-e_i}\delta_f,\]
\[ \de_{t_i}(m \de_t^\alpha \delta_f) = m \de_t^{\alpha+e_i}\delta_f.\]

On $\cD_{X\times \A^r_t}$ there is a decreasing $\Z$-indexed filtration
\[ V^k \cD_{X\times \A^r_t} = \bigoplus_{|\alpha| - |\beta|\geq k} \cD_X t^\alpha \de_t^\beta,\]
which is multiplicative and which satisfies, for example, $t_i \in V^1\cD_{X\times \A^r_t}, \de_{t_j} \in V^{-1}\cD_{X\times \A^r_t}$ and 
\[ V^0 \cD_{X\times \A^r_t} = \cD_X\langle t_1,\dots, t_r, t_i \de_{t_j} \mid 1\leq i,j\leq r\rangle.\]

The Kashiwara-Malgrange $V$-filtration (here, following Saito, indexed by $\Q$) on $\Gamma_+(\cM)$ along $t_1,\dots, t_r$ is the unique decreasing, exhaustive filtration $(V^\chi \Gamma_+(\cM))_{\chi \in \Q}$ which is discrete\footnote{This means that there is an increasing sequence $\alpha_j \in \Q$ with $\lim_{j\to -\infty} \alpha_j = -\infty$ and $\lim_{j\to \infty} \alpha_j = \infty$, such that $V^\chi \Gamma_+(\cM)$ for $\chi \in (\alpha_j,\alpha_{j+1})$ only depends on $j$.} and left continuous\footnote{Meaning $V^\chi \Gamma_+( \cM )= \bigcap_{\beta < \chi}V^\beta\Gamma_+(\cM)$. In other words, $V^\chi \Gamma_+( \cM)$ is constant for $\chi \in (\alpha_j,\alpha_{j+1}]$.}, satisfying the following:
\begin{enumerate} 
\item For all $\chi \in \Q, k\in \Z$, we have $V^k \cD_{X\times \A^r} \cdot V^\chi \Gamma_+(\cM) \subseteq V^{\chi+k} \Gamma_+(\cM)$, with equality if $\chi \gg 0$ and $k\geq 0$,
\item For all $\chi\in \Q$, $V^\chi \Gamma_+(\cM)$ is $V^0\cD_{X\times \A^r_t}$-coherent,
\item The action of $s+\chi$ on ${\rm gr}_V^\chi(\Gamma_+(\cM))$ is nilpotent, where $s = - \sum_{i=1}^r \de_{t_i} t_i$ and 
\[ {\rm gr}_V^\chi(\Gamma_+(\cM)) = V^\chi \Gamma_+(\cM) / V^{>\chi}\Gamma_+(\cM), \quad V^{>\chi} \Gamma_+(\cM) = \bigcup_{\beta > \chi} V^\beta \Gamma_+(\cM).\]
\end{enumerate}

When $r=1$, the $V$-filtration is related to nearby and vanishing cycles under the Riemann-Hilbert correspondence. For this reason, we write
\[ \psi_{f,\lambda}(\cM) = \psi_{t,\lambda}(\Gamma_+(\cM)) = {\rm gr}_V^{\lambda}(\Gamma_+(\cM)) \text{ for } \lambda \in (0,1],\]
and
\[ \phi_{f,\lambda}(\cM) = \psi_{f,\lambda}(\cM) \text{ for } \lambda \in (0,1),\quad \phi_{f,1}(\cM) = \phi_{t,1}(\Gamma_+(\cM)) = {\rm gr}_V^0 (\Gamma_+(\cM)).\]

These have nilpotent endomorphisms $N = s+\lambda$. Moreover, $t$ and $\de_t$ give maps
\[ {\rm var}\colon\phi_{f,1}(\cM) \xrightarrow[]{t} \psi_{f,1}(\cM),\]
\[ {\rm can} \colon \psi_{f,1}(\cM) \xrightarrow[]{\de_t} \phi_{f,1}(\cM).\]

If $(\cM,F,W)$ underlies a mixed Hodge module on $X$, then when $r=1$, the Hodge filtration on nearby cycles is unshifted and the filtration on unipotent vanishing cycles is shifted. We write this succinctly as
\[ F_p = F_{p+1-\lceil \lambda\rceil} {\rm gr}_V^{\lambda}(\Gamma_+(\cM)).\]

For $r \geq 1$, and for any $\chi \in \Q$, the module ${\rm gr}_V^\chi(\Gamma_+(\cM))$ has its \emph{relative monodromy filtration} with respect to the nilpotent operator $N = s+\chi$ and the filtration $L_\bullet {\rm gr}_V^\chi(\Gamma_+(\cM)) = {\rm gr}_V^\chi(W_\bullet \Gamma_+(\cM))$ (however, see the convention \ref{eq-relMonoShift} for shifting the filtration $L_\bullet$ when considering nearby cycles below). We denote this filtration by $W_\bullet {\rm gr}_V^\chi(\Gamma_+(\cM))$. It is characterized by the following properties:
\[ N W_\bullet {\rm gr}_V^\chi(\Gamma_+(\cM)) \subseteq W_{\bullet-2}{\rm gr}_V^\chi(\Gamma_+(\cM)),\]
\[ N^\ell\colon {\rm gr}_W^{\bullet+\ell} {\rm gr}_L^\bullet {\rm gr}_V^\chi(\Gamma_+(\cM)) \to  {\rm gr}_W^{\bullet-\ell} {\rm gr}_L^\bullet {\rm gr}_V^\chi(\Gamma_+(\cM))  \text{ is an isomorphism}.\]

Back to the case $r=1$, the weight filtration is the relative monodromy filtration for $L_\bullet {\rm gr}_V^\lambda(\Gamma_+(\cM))$ with the nilpotent operator $N = s+\lambda$. Here, $L$ is defined by
\begin{equation} \label{eq-relMonoShift} L_k {\rm gr}_V^\lambda(\Gamma_+(\cM)) = {\rm gr}_V^\lambda(W_{k+\lceil \lambda \rceil}\Gamma_+(\cM)).\end{equation}

With these definitions, $(\phi_{f,1}(\cM),F,W)$ and $(\bigoplus_{\lambda \in (0,1]} \psi_{f,\lambda}(\cM),F,W)$ underlie mixed Hodge modules $\phi_{f,1}(M)$ and $\psi_f(M)$ on $X$, which are supported on the hypersurface $\{f=0\}$.

\begin{rmk} \label{rmk-restriction} If $i\colon H = \{f=0\} \hookrightarrow X$ is the closed embedding, then by \cite{SaitoMHM}*{Cor. 2.24}, the ${\rm var}$ and ${\rm can}$ maps represent $i_* i^!$ and $i_* i^*$, respectively, in the category of mixed Hodge modules. Namely, we have quasi-isomorphisms
\[ i_* i^! M \cong [ {\rm var}\colon\phi_{f,1}(M) \to \psi_{f,1}(M)(-1)],\]
\[ i_*i^* M[-1] \cong [{\rm can} \colon\psi_{f,1}(M) \to \phi_{f,1}(M)],\]
where the $(-1)$ is a Tate twist.

In terms of the underlying bi-filtered $\cD$-modules, we have
\[ F_p i_* i^! \cM = [ F_{p+1}{\rm gr}_V^0(\Gamma_+(\cM)) \xrightarrow[]{t} F_{p+1} {\rm gr}_V^1 (\Gamma_+(\cM))],\]
\[ W_k i_* i^! \cM = [ W_k {\rm gr}_V^0(\Gamma_+(\cM)) \xrightarrow[]{t} W_{k-2} {\rm gr}_V^1( \Gamma_+(\cM))],\]
\[ F_p i_* i^* \cM = [F_p {\rm gr}_V^1 (\Gamma_+(\cM)) \xrightarrow[]{\de_t} F_{p+1} {\rm gr}_V^0 (\Gamma_+(\cM))],\]
\[ W_k i_* i^* \cM = [W_k {\rm gr}_V^1 (\Gamma_+(\cM)) \xrightarrow[]{\de_t} W_k{\rm gr}_V^0 (\Gamma_+(\cM))].\]

For related formulas when $r>1$, see \cite{CD}*{Thm. 1.2} and \cite{CDS}*{Thm. 2}.
\end{rmk}

\subsection{Monodromic Modules and Fourier-Laplace Transform} \label{sect-FL} Let $E = X\times \A^r_z$ and consider a coherent monodromic $\cD_E$-module $\cM$. Here monodromic means that $\cM = \bigoplus_{\chi \in \Q} \cM^\chi$, where $\cM^\chi = \bigcup_{j\geq 1} \ker((\theta_z - \chi +r)^j)$ and $\theta_z = \sum_{i=1}^r z_i\de_{z_i}$. It is easy to see that if $V^\bullet \cM$ is the $V$-filtration along $z_1,\dots, z_r$, then
\[ V^\chi \cM = \bigoplus_{\lambda \geq \chi} \cM^\lambda,\]
and so there is a natural isomorphism $\cM^\chi \cong {\rm gr}_V^\chi(\cM)$.

The Fourier-Laplace (or Fourier-Sato) transform of $\cM$ is the $\cD_{E^\vee}$-module ${\rm FL}(\cM)$ which has the same underlying $\cD_X$-module structure, but for which
\[ y_i \cdot m = - \de_{z_i}(m) \text{ and } \de_{y_i} \cdot m = z_i (m).\]

It is easy to see that ${\rm FL}(\cM)$ is monodromic. Indeed,
\[ (\theta_y - \chi +r)(m) = (-\sum_{i=1}^r \de_{z_i}z_i - \chi + r)(m) = - (\theta_z - (r-\chi) +r)(m),\]
and so in fact, we have
\[ {\rm FL}(\cM)^\chi = \cM^{r-\chi}.\]

In \cite{CD}, the author and Chen study the Fourier-Laplace transform for mixed Hodge modules whose underlying $\cD$-module is monodromic (see \cite{MonoMHM1} for an approach using gluing). These are called \emph{monodromic mixed Hodge modules}. Some of the key results concerning monodromic mixed Hodge modules are the following:
\begin{thm}[\cite{CD}*{Thm. 1.5}] \label{thm-MonoMHM} Let $(\cM,F,W)$ underlie a monodromic mixed Hodge module on $E$. Then
\[ F_p \cM = \bigoplus_{\chi \in \Q} F_p \cM^\chi\]
and if $N = \bigoplus_{\chi \in \Q} (\theta_z - \chi + r)$ is the nilpotent operator, then
\[ N W_\bullet \cM \subseteq W_{\bullet-2} \cM.\]
\end{thm}

\begin{rmk} \label{rmk-restrictionMono} Let $(\cM,F,W)$ be a bi-filtered $\cD_E$-module underlying a monodromic mixed Hodge module. Above, we mentioned that there is an isomorphism $\cM^\chi \cong {\rm gr}_V^\chi(\cM)$. By Theorem \ref{thm-MonoMHM}, this is a filtered isomorphism, meaning
\[ F_p \cM^\chi = F_p\cM \cap \cM^\chi \cong F_p {\rm gr}_V^\chi(\cM),\]
\[ W_k \cM^\chi = W_k \cM \cap \cM^\chi \cong W_k {\rm gr}_V^\chi(\cM),\]
where the last term is the relative monodromy filtration of $L_\bullet {\rm gr}_V^\chi(\cM) = {\rm gr}_V^\chi(W_\bullet \cM)$ with respect to $N = \theta - \chi +r$.

When $r=1$, this allows us to re-write the restriction functors in the monodromic setting as follows. We remark that the shifts in the weight filtration are to make up for the shift of the filtration $L$ for nearby cycles.
\[ F_p i^! \cM = [ F_{p+1}\cM^0 \xrightarrow[]{z} F_{p+1}\cM^1],\]
\[ W_k  i^! \cM = [ W_k \cM^0 \xrightarrow[]{z} W_{k-1} \cM^1],\]
\[ F_p i^* \cM = [F_p \cM^1 \xrightarrow[]{\de_z} F_{p+1}\cM^0],\]
\[ W_k i^* \cM = [W_{k+1}\cM^1 \xrightarrow[]{\de_z} W_k\cM^0].\]
\end{rmk}

\begin{rmk} \label{rmk-MonodromicPushforward} If $M$ is a monodromic mixed Hodge module on $E$ with zero section $\sigma\colon X \to E$ and projection $p\colon E \to X$, then there are natural isomorphisms
\[ \sigma^* M \cong p_* M, \quad \sigma^! M \cong p_! M,\]
see, for example, \cite{CD}*{Rem. 3.8} (which follows immediately from \cite{Ginzburg}*{Prop. 10.4}).
\end{rmk}

To study the Fourier-Laplace transform for mixed Hodge modules, one introduces the product variety $\cE = E \times_X E^\vee$ with its projections $p\colon \cE \to E$ and $q\colon\cE \to E^\vee$. As a variety, $\cE = X\times \A^r_z \times \A^r_y$ and it admits a globally defined function $g = \sum_{i=1}^r z_i y_i$, which is the pairing on the fibers of $E$ and $E^\vee$.

One of the main results of \cite{CD} is an explicit isomorphism of $\cD_{E^\vee}$-modules
\begin{equation} \label{eq-FourierIso} {\rm FL}(\cM) \cong \cH^0 \sigma^* \phi_g p^!(\cM)[-r],\end{equation}
where, as $\cD$-modules, $p^!(\cM)[-r] = \cM\boxtimes \cO_{\A^r_y}$. Given a monodromic module $\cM = \bigoplus_{\chi \in \Q} \cM^\chi$, we write $\cM^{\chi + \Z} = \bigoplus_{\ell \in \Z} \cM^{\chi+\ell}$ for the direct summand $\cD$-module. It is easy to see then that $\cM = \bigoplus_{\lambda \in [0,1)} \cM^{\lambda + \Z}$. 

Under the Isomorphism \ref{eq-FourierIso}, the Fourier-Laplace transform inherits a Hodge and weight filtration, because the object on the right hand side underlies a mixed Hodge module. These filtrations are expressed by the following theorem:
\begin{thm}[\cite{CD}*{Thm. 1.4}] \label{thm-FourierHodgeWeight} Let $(\cM,F,W)$ be a bi-filtered $\cD_E$-module underlying a monodromic mixed Hodge module $M$ on $E$. Then, under the isomorphism \ref{eq-FourierIso}, we have
\[ F_p {\rm FL}(\cM)^{r-\chi} = F_{p-\lceil \chi\rceil} \cM^\chi\]
\[ W_k{\rm FL}(\cM)^{\lambda+\Z} = {\rm FL}(W_{k+r+\lceil \lambda\rceil}\cM)^{\lambda + \Z} \text{ for } \lambda \in [0,1),\]
where $W_{k+r+\lceil \lambda \rceil} \cM$ is a monodromic sub-mixed Hodge module of $\cM$, and so ${\rm FL}(-)$ is defined as above.
\end{thm}

\begin{rmk}\label{rmk-FLRestriction} Let $i\colon X \to E$ and $\sigma\colon X \to E^\vee$ be the two zero sections of the trivial line bundles $E$ and $E^\vee$. From these relations, we see that ${\rm FL}$ interchanges the two types of restriction to zero sections. We will write out the relationship when $r=1$, but a similar relationship holds for $r>1$. See Remark \ref{rmk-FLRestrictionFunctorial} below for another argument.

Let $M$ be a monodromic Hodge module on $E$ with underlying bi-filtered $\cD_E$-module $(\cM,F,W)$. Consider ${\rm FL}(M)$, its Fourier-Laplace transform, which is a monodromic Hodge module on $E^\vee$. By Remark \ref{rmk-restrictionMono} above, we see that
\[ F_p  \sigma^* {\rm FL}(\cM) = [ F_p {\rm FL}(\cM)^1 \xrightarrow[]{\de_y} F_{p+1} {\rm FL}(\cM)^0]\]
\[ W_k  \sigma^*{\rm FL}(\cM) = [ W_{k+1} {\rm FL}(\cM)^1 \xrightarrow[]{\de_y} W_k {\rm FL}(\cM)^0].\]

But by Theorem \ref{thm-FourierHodgeWeight}, this is the same thing as
\[ [F_p \cM^0 \xrightarrow[]{z} F_p\cM^1] = F_{p-1}  i^!(\cM)[1],\]
\[ [ W_{k+2}\cM^ 0 \xrightarrow[]{z} W_{k+1} \cM^1] = W_{k+2} i^!(\cM)[1]\]

So we see the relation $ \sigma^* \circ {\rm FL} \cong i^! (1)[1]$. Similarly, one can check that $\sigma^! \circ {\rm FL} \cong i^*[-1]$.
\end{rmk}

\begin{rmk} \label{rmk-FLZeroConstant} At the $\cD$-module level, it is clear that modules supported on the zero section $X\times \{0\} \subseteq E$ correspond under the Fourier-Laplace transform with ``constant" $\cD$-modules, i.e., modules for which every element is $\de_{y_1},\dots, \de_{y_r}$-torsion. It is not hard to see that such modules are of the form $\cN \boxtimes \cO_{\A^r_y}$ for some $\cD_X$-module $\cN$.

Of course, this is also true at the Hodge module level. Indeed, let $i\colon X \to E$ be the inclusion of the zero section. We have a Cartesian diagram
\[ \begin{tikzcd} X \times \{0\} \times \A^r_y \ar[r,"I"] \ar[d,"P"] & X\times  \A^r_z \times \A^r_y\ar[d,"p"]\\ X\times \{0\} \ar[r,"i"] & X\times \A^r_z \end{tikzcd}.\]

Then
\[ {\rm FL}(i_* N) = q_* \phi_g p^! i_* N [-r] \cong q_* \phi_g I_* P^! N[-r],\]
where the last isomorphism holds by Smooth Base Change \ref{eg-BaseChange}. Now, $X\times \{0\} \times \A^r_y \subseteq \{g=0\}$, and so $\phi_g I_* = I_*$. Hence, we get
\[ {\rm FL}(i_* N) = (q\circ I)_* P^! N[-r],\]
and $q\circ I$ is the identity $X\times \{0\} \times \A^r_y \to X\times \A^r_y$. We see then that
\[ {\rm FL}(i_* N) \cong \pi^! N [-r],\]
where $\pi\colon X\times \A^r_y \to X$ is the projection.
\end{rmk}

\section{Comments on the Fourier-Laplace Transform} \label{sect-FL}
This section is devoted to a proof of Proposition \ref{prop-main}. To do this, we first prove Lemma \ref{lem-supportVanCycles}.

These results only concern the underlying $\cD$-modules, and so we can replace $p^!(\cM)[-r]$ with $\cM[y_1,\dots,y_r]$. Recall that we have an isomorphism
\[ {\rm FL}(\cM) \cong \cH^0 \sigma^* \phi_g(\cM[y_1,\dots, y_r]),\]

The better behaved object would be
\[ \sigma^* \phi_g(\cM[y_1,\dots, y_r]),\]
which, a priori, is an object in the derived category. 

The first result of this section is the following:
\begin{prop} \label{prop-main} Let $\cM$ be a monodromic regular holonomic $\cD_E$-module. Then
\[ \cH^i \sigma^* \phi_g(\cM[y_1,\dots, y_r]) = 0 \text{ for all } i \neq 0.\]
\end{prop}

In fact, this is a direct consequence of the following result (see also \cite{Davison}*{Lem A.4})
\begin{lem} \label{lem-supportVanCycles}With the notation above, the module $\phi_g(\cM[y_1,\dots, y_r])$ is supported on $\sigma(E^\vee) = X\times \{0\} \times \A^r_y$.
\end{lem}
\begin{proof} To approximate the support of $\phi_g(-)$, using \cite[Prop. 4.20]{MaximVanCycles} we must find a Whitney stratification of $\cE$ so that $\cM[y_1,\dots, y_r]$ is locally constant along that Whitney stratification. Clearly, if $\{S_j\}$ is a Whitney stratification for $E$ so that $\cM$ is locally constant along that stratification, then $\{S_j \times \A^r_y\}$ will be one such stratification for $\cE$.

The formula for the support of the vanishing cycles is 
\[ {\rm Supp}(\phi_g(\cM[y_1,\dots,y_r])) \subseteq \bigcup_j {\rm Sing}_{S_j \times \A^r_y}(g\vert_{S_j \times \A^r_y}).\]

Note that, after choosing coordinates $s_1,\dots, s_{d_j}$ on $S_j$, we get coordinates $s_1,\dots, s_{d_j},y_1,\dots, y_r$ on $S_j \times \A^r_y$. Let $g_j = g\vert_{S_j \times \A^r_y}$. Then the singular locus is defined by the vanishing of the partial derivatives, i.e.,
\[ \de_{s_1}(g_j),\dots, \de_{s_{d_j}}(g_j), \de_{y_1}(g_j),\dots, \de_{y_r}(g_j),\]
but the last few such partial derivatives are clearly $z_1,\dots, z_r$. Hence, the Jacobian ideal of each $g_j$ contains the ideal $(z_1,\dots, z_r)$, and so we see that
\[ {\rm Sing}_{S_j\times \A^r_y}(g_j) \subseteq X\times \{0\} \times \A^r_y,\]
which proves the claim.
\end{proof}

\begin{proof}[Proof of Proposition \ref{prop-main}] If $\phi_g(\cM[y_1,\dots, y_r])$ is supported on $\sigma(E^\vee)$, it can be written
\[ \phi_g(\cM[y_1,\dots,y_r]) = \sigma_+ \cN,\]
for some module $\cN$ on $E^\vee$. But then $\sigma^* \sigma_+ \cN = \cN$, and so the desired vanishing is clear.
\end{proof}

The result on the support of the vanishing cycles in this case also leads to the following, which should be compared with \cite{KashShap}*{Prop. 10.1.13}.

\begin{cor} \label{cor-lastFunctor} For $\cM$ a monodromic regular holonomic $\cD_E$-module, we have
\[ {\rm FL}_X(\cM) \cong \sigma^! \phi_g(\cM[y_1,\dots, y_r]),\]
and in particular, 
\[ \cH^i \sigma^! \phi_g(\cM[y_1,\dots, y_r]) = 0 \text{ for all } i \neq 0.\]
\end{cor}

Thanks to monodromicity of $\phi_g(\cM[y_1,\dots, y_r])$ (which ensures $\sigma^* \cong q_*$ and $\sigma^! \cong q_!$, see Remark \ref{rmk-MonodromicPushforward}), the left-most functor in defining the Fourier-Laplace transform can actually be one of four: namely, we have isomorphisms
\[ {\rm FL}_X(\cM) \cong \sigma^*\phi_g(\cM[y_1,\dots, y_r]) \cong q_*\phi_g(\cM[y_1,\dots, y_r])\] \[\cong \sigma^!\phi_g(\cM[y_1,\dots, y_r]) \cong q_!\phi_g(\cM[y_1,\dots, y_r]).\]

We see easily that the Fourier-Laplace transform commutes with proper push-forwards at the $\cD$-module level, and the following result shows that this is true at the level of monodromic mixed Hodge modules, too.

\begin{cor} \label{cor-ProperFL} Let $f\colon X\to Y$ be a proper morphism and $F = f \times {\rm id}_{\A^r}\colon X\times \A^r_z \to Y \times \A^r_z$. Let $M$ be a monodromic mixed Hodge module on $X\times \A^r_z$. Then $\cH^i F_*(M)$ is monodromic for all $i \in \Z$.

Moreover, if ${\rm FL}_X$, ${\rm FL}_Y$ are the respective Fourier-Laplace transforms for monodromic modules on $X \times \A^r_z$ and $Y\times \A^r_z$, and if $\hat{F} = f\times {\rm id}_{\A^r}\colon X \times \A^r_y \to Y\times \A^r_y$, then
\[ \hat{F}_* {\rm FL}_X(M) \cong {\rm FL}_Y F_*(M)\]
in the derived category of monodromic mixed Hodge modules on $Y\times \A^r_y$.
\end{cor}
\begin{proof} We have a Cartesian diagram
\[ \begin{tikzcd} X \times \A_y^n \ar[r,"\sigma_X"] \ar[d,"\hat{F}"] & X\times \A_z^r\times \A_y^r \ar[d,"P"] \\ Y\times \A_y^r \ar[r,"\sigma_Y"] & Y\times \A_z^r\times \A_y^r \end{tikzcd}\]
hence, by Proper Base Change as in Example \ref{eg-BaseChange}, using that $P_* = P_!$, we get
\[ \hat{F}_* \circ \sigma_X^* \cong \sigma_Y^* \circ  P_*.\]

Now, we have that $g_X\colon X\times \A^r_z \times \A^r_y \to \A^1$ is actually equal to $g_X = g_Y\circ P$, by definition. So by proper push-forward compatibility with vanishing cycles \cite[Thm. 2.14]{SaitoMHM}, we get
\[P_* \circ \phi_{g_X} \cong \phi_{g_Y} \circ P_*.\]

Finally, we have a Cartesian diagram
\[ \begin{tikzcd} X\times \A_z^r \times \A_y^r \ar[r,"q_X"] \ar[d,"P"] & X\times \A_z^r \ar[d,"F"] \\ Y \times \A_z^r \times \A_y^r \ar[r,"q_Y"] & Y \times \A_z^r \end{tikzcd},\]
and so $P_*\circ q_X^! \cong q_Y^!\circ F_*$, again by Proper Base Change \ref{eg-BaseChange}.

Putting this together, we have
\[ \widehat{F}_* {\rm FL}_X = \widehat{F}_* \sigma_X^* \phi_{g_X} q_X^! \cong \sigma_Y^* P_* \phi_{g_X} q_X^! \cong \sigma_Y^* \phi_{g_Y} P_* q_X^! \cong \sigma_Y^* \phi_{g_Y} q_Y^! F_* = {\rm FL}_Y F_*.\]
\end{proof}

If $M$ is a \emph{unipotent} monodromic module, i.e., $\cM = \cM^{\Z}$, then it is easy to see that ${\rm FL}(M)$ is also unipotent. In particular, $\sigma^* \phi_{g,\lambda} p^!(M)[-r] = 0$ for all $\lambda \neq 1$. Hence, for such modules we can consider the \emph{unipotent Fourier-Laplace transform} ${\rm FL}_{\rm unip} = \sigma^* \phi_{g,1} p^! [-r]$. 

Recall that we have defined in \cite{CD} the inverse Fourier transform functor $\overline{\rm FL}$, which is simply a Tate-twisted version of ${\rm FL}$, though the twist is different on the unipotent and non-unipotent parts. In the unipotent case, the \emph{Fourier-Inversion formula} \cite{CD}*{Cor. 1.6} gives
\begin{equation} \label{eq-FLInv}  {\rm FL}_{\rm unip} {\rm FL}_{\rm unip} (-) \cong a_*(-)(r), \end{equation}
where $a\colon X \times \A^r_z \to X\times \A^r_z$ is the antipode involution $(x,z) \mapsto (x,-z)$. Moreover, we have
\[ \mathbf D_{E^\vee} \circ {\rm FL}_{\rm unip} \cong ({\rm FL}_{\rm unip} \circ \mathbf D_E)(-r).\]

\begin{rmk}\label{rmk-FLZeroConstant2} We saw in Remark \ref{rmk-FLZeroConstant} that if $i\colon X \to X\times \A^r_z$ is the inclusion of the zero section, there is a natural isomorphism
\[ {\rm FL}(i_* N) \cong {\rm FL}_{\rm unip} (i_*N) \cong \pi^! N[-r],\]
where $\pi\colon X\times \A^r_y \to X$ is the projection.

By Fourier Inversion, we get
\[ {\rm FL}_{\rm unip}(\pi^! N[-r]) \cong a_* i_*N (r) = i_* N (r),\]
using the fact that $a \circ \sigma = \sigma$. As $\pi^! N[-r]$ is isomorphic to $\pi^* N(r)[r]$, we get
\[ {\rm FL}_{\rm unip}(\pi^* N[r]) \cong i_* N.\]
\end{rmk}

\begin{rmk} \label{rmk-FLRestrictionFunctorial} Let $\sigma\colon X \to X\times \A^r_y$ and $i\colon X\to X\times \A^r_z$ be the inclusions of the zero sections.

Using Fourier inversion, Remark \ref{rmk-MonodromicPushforward} and Remark \ref{rmk-FLZeroConstant2}, it is easy to check that 
\[ \sigma^* \circ {\rm FL}(-r)[-r] \text{ is right adjoint to } i_*\]
and
\[ \sigma^! \circ {\rm FL}[r] \text{ is left adjoint to } i_*.\]

This gives isomorphisms
\[ \sigma^* \circ {\rm FL}(-r)[-r] \cong i^!, \quad \sigma^! \circ {\rm FL}[r] \cong i^*.\]

Note that $\sigma^* \circ {\rm FL} = \sigma^* \circ {\rm FL}_{\rm unip}$ and similarly for $\sigma^!$, as the other monodromic pieces play no role in the restriction formula.

We will check the former adjunction, to see why the Tate twist arises. The latter also follows from this by duality.

We have by Remark \ref{rmk-MonodromicPushforward} the isomorphism
\[ {\rm Hom}(N, \sigma^* {\rm FL}_{\rm unip}(M)(-r)[-r]) \cong {\rm Hom}(N,\pi_* {\rm FL}_{\rm unip}(M)(-r)[-r]) \cong {\rm Hom}(\pi^* N[r], {\rm FL}_{\rm unip}(M)(-r)),\]
which, by applying the equivalence $a_* {\rm FL}_{\rm unip}$ to both arguments, gives
\[ {\rm Hom}(a_* {\rm FL}_{\rm unip}(\pi^* N[r]), M) \cong {\rm Hom}(i_* N, M),\]
where the last isomorphism follows from Remark \ref{rmk-FLZeroConstant2}.
\end{rmk}

\section{Microlocalization for mixed Hodge modules} \label{sect-Microlocal} This section is devoted to a proof of Theorem \ref{thm-Microlocalization}. We recall the set-up from the introduction. Let $Z \subseteq X$ be defined by $f_1,\dots, f_r \in \cO_X(X)$ and let $\Gamma\colon X \to X\times \A^r_t$ be the graph embedding along $f_1,\dots, f_r$, with coordinates $t_1,\dots,t_r$ on the $\A^r_t$ term. We let $X\times \A^r_z$ be the normal bundle of the zero section of $X\times \A^r_t$. Let $X\times \A^r_y$ be the dual bundle to $X \times \A^r_z$.

Let $M$ be a mixed Hodge module on $X$ with $(\cM,F,W)$ its underlying bi-filtered left $\cD_X$-module. Then $\Gamma_* M$ has underlying left $\cD_X$-module $\Gamma_+(\cM) = \bigoplus_{\alpha \in \N^r} \cM \de_t^\alpha \delta_f$, with weight filtration
\[ W_\bullet \Gamma_+(\cM) = \bigoplus_{\alpha \in \N^r} (W_\bullet \cM)\de_t^\alpha \delta_f,\]
Hodge filtration
\[ F_\bullet \Gamma_+(\cM) = \bigoplus_{\alpha \in \N^r} F_{\bullet - |\alpha| - r}\cM \de_t^\alpha \delta_f,\]
and with $\Q$-indexed $V$-filtration along $t_1,\dots,t_r$, denoted $V^\bullet\Gamma_+(\cM)$.

The Verdier specialization of $\Gamma_* M$ is a monodromic mixed Hodge module ${\rm Sp}_Z(M) = {\rm Sp}(\Gamma_* M)$ on $X\times \A^r_z$. For details of this construction, see \cite{BMS}*{Sect. 1.3} and \cite{CD}*{Sect. 2.4}. Its underlying filtered $\cD$-module is
\[ {\rm Sp}_Z(\cM) = \bigoplus_{\chi \in \Q} {\rm gr}_V^\chi(\cM),\]
\[ F_p {\rm Sp}_Z(\cM) = \bigoplus_{\chi \in \Q} F_p {\rm gr}_V^\chi(\cM).\]

The weight filtration on ${\rm Sp}_Z(\cM)$ is given by
\[ W_\bullet {\rm Sp}_Z(\cM) = \bigoplus_{\chi \in \Q} W_\bullet {\rm gr}_V^\chi(\cM),\]
where $W_\bullet {\rm gr}_V^{\chi}(\cM)$ is the relative monodromy filtration along the nilpotent operator $N = s+\chi$ as defined above.

\begin{rmk} \label{rmk-sameRestriction} As argued in \cite{SaitoMHM}*{Pg. 269}, if $i\colon Z \subseteq X$ is the inclusion and $\sigma\colon X \to X\times \A^r_z$ is the zero section, then there are natural isomorphisms
\[ \sigma_*\sigma^!{\rm Sp}_Z(\cM) \cong i_* i^! \cM, \quad \sigma_* \sigma^* {\rm Sp}_Z(\cM) \cong i_* i^* \cM,\]
which are also easy (at least when $r=1$) to see using the formulas for restriction given in Example \ref{rmk-restriction} above.
\end{rmk}

Finally, we want to compute $\mu_Z(M) = {\rm FL} {\rm Sp}_Z(M)$, and so we need to take the Fourier-Laplace transform. As stated in Theorem \ref{thm-FourierHodgeWeight}, the Hodge and weight filtrations are given by
\[ F_\bullet {\rm FL} {\rm Sp}_Z(\cM)^{r-\chi} = F_{\bullet-\lceil \chi\rceil} {\rm gr}_V^{\chi}(\cM).\]
\[ W_\bullet {\rm FL}{\rm Sp}_Z(\cM)^{\lambda + \Z} = {\rm FL}_X(W_{\bullet + r+ \lceil \lambda\rceil} {\rm Sp}_Z(\cM))^{\lambda + \Z}.\]

On the other hand, if $\pi \colon X\times \A^r_y \to X$ is the projection, we can consider $\pi^!(M)[-r]$, which has underlying $\cD$-module $\cM[y_1,\dots, y_r]$. The Hodge filtration is defined by
\[ F_\bullet (\cM[y_1,\dots, y_r]) = (F_{\bullet-r} \cM)[y_1,\dots, y_r],\]
and the weight filtration is given by
\[ W_\bullet(\cM[y_1,\dots,y_r]) = (W_{\bullet +r} \cM)[y_1,\dots,y_r].\]

Let $g = \sum_{i=1}^r y_i f_i \in \cO_{X\times \A^r_y}(X\times \A^r_y)$ and consider the graph embedding $\gamma\colon X \times \A^r_y \to X\times \A^r_y \times \A^1_\xi$ along $g$. We have $\gamma_+(\cM[y_1,\dots, y_r]) = \cM[y_1,\dots, y_r,\de_\xi]\delta_g$. We also consider the \emph{partial microlocalization} $\cM_g =\gamma_+(\cM[y_1,\dots, y_r])[\de_\xi^{-1}] = \cM[y_1,\dots, y_r,\de_\xi^{\pm 1}]\delta_g$, as in \cite{ThomSebastianiHodge}.

The $\cD$-action is given for all $m\in \cM$, $\alpha \in \N^r, j \in \Z$ by
\[ \de_{x_i}(m y^\alpha \de_{\xi}^j\delta_g) = \de_{x_i}(m)y^\alpha \de_{\xi}^j \delta_g- \left(\sum_{\ell=1}^r \de_{x_i}(f_\ell) m y^{\alpha+e_\ell}\right) \de_{\xi}^{j+1}\delta_g \text{ for all } 1\leq i \leq \dim X,\]
\[ \de_{y_i}(m y^\alpha \de_{\xi}^j\delta_g) = \alpha_i m y^{\alpha-e_i} \de_{\xi}^j\delta_g - f_j m y^\alpha \de_\xi^{j+1}\delta_g \text{ for all } 1\leq i\leq r,\]
\[ \de_{\xi}(m y^\alpha \de_{\xi}^j\delta_g) = m y^\alpha \de_{\xi}^{j+1}\delta_g,\]
\[ \xi(m y^\alpha \de_{\xi}^j\delta_g) = g m y^\alpha \de_\xi^j\delta_g - j m y^\alpha \de_\xi^{j-1}\delta_g.\]

The Hodge filtration is given by
\[ F_\bullet \cM_g= \bigoplus_{j \in \Z} (F_{\bullet-j-r-1} \cM)[y_1,\dots, y_r] \de_\xi^j \delta_g,\]
and the weight filtration is
\[ W_\bullet \cM_g = (W_{\bullet+r}\cM)[y_1,\dots, y_r,\de_{\xi}^{\pm 1}] \delta_g.\]

Let $s = -\de_{\xi} \xi$. Note that, for $\theta_y = \sum_{i=1}^r y_i \de_{y_i}$, we have
\[ \theta_y( m y^\alpha \de_\xi^j \delta_g) = |\alpha| m y^\alpha\de_\xi^j\delta_g - g m y^\alpha \de_\xi^{j+1}\delta_g = (|\alpha| + s-j)(m y^\alpha \de_\xi^j \delta_g).\]

Hence, elements of the form $m y^\alpha \de_\xi^j \delta_g$ are eigenvectors of the operator $\theta_y - s$, with eigenvalue $|\alpha|-j$. We can write
\[ \cM_g = \bigoplus_{\ell \in \Z} E^{(\ell)},\]
where
\[ E^{(\ell)} = \ker(\theta_y - s - \ell) = \bigoplus_{|\alpha| = j+\ell} \cM y^\alpha \de_\xi^j \delta_g.\]

Let $V^\bullet \cM_g$ be the microlocal $V$-filtration along $\xi$, defined in \cite{ThomSebastianiHodge}*{Sect. 1.1} (see \cite{SaitoMicrolocal} for the case $\cM = \cO_X$). 

By \cite{ThomSebastianiHodge}*{(1.1.8)}, we have that
\[ \de_\xi^j(F_p V^\lambda \cM_g) = F_{p+j} V^{\lambda-j} \cM_g\]
and by \cite{ThomSebastianiHodge}*{(1.1.9)} that the inclusion $\cM[y_1,\dots, y_r,\de_\xi]\delta_g \subseteq \cM_g$ induces an isomorphism for all $\lambda <1$ and all $p\in \Z$,
\begin{equation} \label{eq-MicrolocalSameGr} F_p {\rm gr}_V^\lambda(\cM[y_1,\dots, y_r,\de_\xi]\delta_g) \cong F_p {\rm gr}_V^{\lambda}(\cM_g).\end{equation}

Each $V^\lambda \cM_g$ is stable by $\theta_y - s$, so each of them decomposes into eigenspaces. It is clear that the Hodge filtration pieces also decompose into eigenspaces. Let $V^\lambda E^{(\ell)} = E^{(\ell)}\cap V^\lambda$ and $F_p E^{(\ell)} = E^{(\ell)} \cap F_p$. The maps $\de_\xi^j$ move the eigenspaces in a predictable way, and hence we have
\[ \de_\xi^j(F_p V^\lambda E^{(\ell)}) = F_{p+j} V^{\lambda-j} E^{(\ell-j)}\]
for all $j,p,\ell \in \Z$ and $\lambda \in \Q$.

The morphism we define below differs from that in \cite{CDMO}*{Prop. 3.2} and in \cite{D-Microlocal} by the antipodal map on $X\times \A^r_y$ (i.e., interchanging $y_i$ with $-y_i$ and $\de_{y_i}$ with $-\de_{y_i}$). We use the map below because it relates more naturally to the Fourier-Laplace transform.

\begin{prop} \label{prop-varphiProperties} Using the notation above, define $\varphi\colon\cM_g \to \Gamma_+(\cM)$ as the unique $\cO_X$-linear map sending
\[ m y^\alpha \de_\xi^j \delta_g \mapsto (-1)^{|\alpha|+j} m \de_t^\alpha \delta_f\]
for all $m \in \cM, \alpha \in \N^r$ and $j \in \Z$.

Then we have the following:
\begin{enumerate} \item $\varphi$ is $\cD_X$-linear.
\item $\varphi \circ \de_{\xi}^k = (-1)^k \varphi$ for all $k\in \Z$,
\item $\varphi \circ y_i = - \de_{t_i} \circ \varphi$ for all $i \in \{1,\dots, r\}$,
\item $\varphi \circ \de_{y_i} = t_i \circ \varphi$ for all $i\in \{1,\dots, r\}$,
\item $\varphi \circ \theta_y = s \circ \varphi$,
\item For any $\ell \in \Z$, the map $\varphi\vert_{E^{(\ell)}}$ is an isomorphism which satisfies
\[ \varphi \circ s = (s - \ell) \circ \varphi.\]
\end{enumerate}
\end{prop}
\begin{proof} These properties follow from some simple computations involving the $\cD$-module structures on both sides. The computations are exactly the same as when $M = \Q_X^H[\dim X]$, so we do not repeat them here. See \cite{CDMO}*{Prop. 3.2} for details.
\end{proof}

From this, we see that the map $\varphi$ relates the Hodge, weight and $V$-filtrations on both sides. Here, because we want an isomorphism of $\cD_X$-modules, we restrict to a single eigenspace of the operator $\theta_y - s$.

\begin{prop} \label{prop-microlocalHodgeVW} The map $\varphi$ satisfies the following properties with respect to the $V$, Hodge and weight filtrations:
\[ \varphi(V^{\lambda} E^{(\ell)}) = V^{\lambda-\ell} \Gamma_+(\cM),\]
\[ \varphi(F_p E^{(\ell)}) = F_{p+\ell-1} \Gamma_+(\cM),\]
\[ \varphi(W_k E^{(\ell)}) = W_{k+r} \Gamma_+(\cM).\]
\end{prop}
\begin{proof} The proof for the $V$-filtrations goes through in exactly the same way as in \cite{CDMO}*{Thm. 3.3}, by replacing $\cO_X$ with $\cM$. We do not repeat that argument here: the main point is to follow the argument for uniqueness of $V$-filtrations.

The claims on the Hodge and weight filtrations follow by keeping track of indices. An arbitrary element of $E^{(\ell)}$ is of the form
\[ m =  \sum_{\alpha \in \N^r} m_\alpha y^\alpha \de_\xi^{|\alpha|-\ell} \delta_g,\]
which maps to $\sum_{\alpha \in \N^r} (-1)^{\ell} m_\alpha \de_t^\alpha \delta_f$ in $\Gamma_+(\cM)$.

We have $m \in F_p E^{(\ell)}$ if and only if $m_\alpha \in F_{p+\ell -|\alpha|-r-1}\cM$ for all $\alpha$. Hence, this maps into
\[ \bigoplus_{\alpha \in \N^r} F_{p+\ell-|\alpha|-r-1}\cM \de_t^\alpha \delta_f = F_{p+\ell - 1} \Gamma_+(\cM).\]

Similarly, $m\in W_k E^{(\ell)}$ if and only if $m_\alpha \in W_{k+r}\cM$ for all $\alpha$, and so this maps to
\[ \bigoplus_{\alpha \in \N^r} W_{k+r} \cM \de_t^\alpha \delta_f = W_{k + r}\Gamma_+(\cM).\]
\end{proof}

We can now prove the first main result of this section, relating $\phi_g(\pi^!(M))[-r]$ and $\mu_Z(M)$.

\begin{proof}[Proof of Theorem \ref{thm-Microlocalization}] By Equation \ref{eq-MicrolocalSameGr}, it suffices to use $\cM_g$ in place of $\cM[y_1,\dots, y_r,\de_\xi]\delta_g$. We have by definition \[ F_p L_{k} \phi_{\xi,\lambda}\cM_g = F_{p+1-\lceil \lambda\rceil} L_k {\rm gr}_V^{\lambda}(\cM_g) = F_{p+1-\lceil \lambda \rceil} {\rm gr}_V^{\lambda}(W_{k+\lceil \lambda \rceil} \cM_g).\]

Decomposing into eigenspaces for $\theta_y -s$, we get that this is equal to
\[ \bigoplus_{\ell \in \Z} F_{p+1-\lceil \lambda \rceil}{\rm gr}_V^{\lambda}(W_{k+\lceil \lambda \rceil} \cM_g)^{(\ell)}.\]

Note that Proposition \ref{prop-microlocalHodgeVW} tells us that $\varphi$ induces an isomorphism
\[ F_{p+1-\lceil \lambda \rceil} {\rm gr}_V^{\lambda}(W_{k+\lceil \lambda \rceil} \cM_g)^{(\ell)} \cong F_{p+\ell -\lceil \lambda\rceil} {\rm gr}_V^{\lambda - \ell}(W_{k+r+ \lceil \lambda \rceil} \Gamma_+(\cM)).\]

As $\varphi$ commutes with $N$, the relative monodromy filtrations are isomorphic to each other (up to the corresponding shift). But this shift is precisely the one we see in the Fourier transform as in Theorem \ref{thm-FourierHodgeWeight} (if $\chi = \lambda -\ell$, then $\lceil \chi \rceil = \lceil \lambda \rceil - \ell$). It is clear by the definition of $\varphi$ that the isomorphism we have defined is $\cD$-linear, using for example the fact that $\varphi \circ y_i = (- \de_{t_i}) \circ \varphi$, which is how the action is defined on the Fourier transform.
\end{proof}

We end this section with a proof of Theorem \ref{thm-MHMLocalization}. We consider the localization triangle
\[ \sigma_* \sigma^! \pi^! \to \pi^!  \to j_* j^* \pi^! \xrightarrow[]{+1},\]
and we apply $\sigma^! \phi_{g,1}$ to it. 

\begin{proof}[Proof of Theorem \ref{thm-MHMLocalization}] Let $\widetilde{i}_Z\colon Z \times \A^r_y \to X \times \A^r_y$ be the natural inclusion. The proof of \cite{Davison}*{Thm. A.1} shows the vanishing
\[ \pi_! \widetilde{i}_{Z*} \widetilde{i}_{Z}^* \psi_{g,1} \pi^* = 0,\]
equivalently,
\[ \pi_! \widetilde{i}_{Z*} \widetilde{i}_{Z}^* \psi_{g,1} \pi^! = 0,\]
as a functor on bounded constructible complexes. As the functor 
\[{\rm rat}\colon D^b({\rm MHM}(X)) \to D^b_c(X^{\rm an})\] is faithful, this vanishing holds for mixed Hodge modules, too. 

Let $i_g\colon\{g =0\} \subseteq X\times \A^r_y$ be the closed embedding, so we have a morphism between exact triangles
\[ \begin{tikzcd} \psi_{g,1} \sigma_* \sigma^! \pi^! \ar[r] \ar[d] & \phi_{g,1} \sigma_* \sigma^! \pi^! \ar[r] \ar[d] & i_{g*}i_g^* \sigma_* \sigma^! \pi^! \ar[d] \ar[r,"+1"] & .\\ \psi_{g,1} \pi^! \ar[r] & \phi_{g,1}  \pi^! \ar[r] & i_{g*}i_g^* \pi^! \ar[r,"+1"] & . \end{tikzcd}.\]

As $\sigma(X) \subseteq \{g=0\}$, the top left term is 0 and so the second arrow in the top row is an isomorphism. For the same reason, the top right term is simply $\sigma_* \sigma^! \pi^!$. If we apply $\pi_! \widetilde{i}_{Z*} \widetilde{i}_Z^*$ to the bottom row, then the left-most term vanishes, and so we get an isomorphism
\[ \pi_! \widetilde{i}_{Z*} \widetilde{i}_Z^* \phi_{g,1} \pi^! \cong \pi_! \widetilde{i}_{Z*}\widetilde{i}_Z^* \pi^!.\]

Note that by the same argument as in the proof of Lemma \ref{lem-supportVanCycles} or by \cite{Davison}*{Lem. A.4}, we have ${\rm supp}(\phi_g \pi^!(M)) \subseteq Z \times \A^r_y$ , and so we have a natural isomorphism
\[ \pi_! \phi_{g,1} \pi^! \cong  \pi_! \widetilde{i}_{Z*}\widetilde{i}_Z^* \phi_{g,1} \pi^!.\]

This gives a commutative diagram
\[ \begin{tikzcd} \pi_! \phi_{g,1} \sigma_* \sigma^! \pi^! \ar[r,"\cong"] \ar[d] & \pi_! i_{g*}i_g^* \sigma_* \sigma^! \pi^! \ar[d] & \pi_! \sigma_* \sigma^! \pi^! \ar[l,swap,"\cong"] \ar[d,"\cong"]\\
\pi_! \phi_{g,1} \pi^! \ar[r] \ar[d,"\cong"] & \pi_! i_{g*}i_{g}^* \pi^! \ar[d] & \pi_! \pi^! \ar[l,swap] \ar[d] \\
\pi_! \widetilde{i}_{Z*}\widetilde{i}_Z^* \phi_{g,1} \pi^! \ar[r,"\cong"] & \pi_! \widetilde{i}_{Z*}\widetilde{i}_Z^* i_{g*}i_g^* \pi^! & \pi_! \widetilde{i}_{Z*} \widetilde{i}_Z^* \pi^! \ar[l,swap,"\cong"]
\end{tikzcd}\]
where maps from the rightmost column to the central column are from the adjunction of $(i_g^*,i_{g*})$. The first row of vertical morphisms come from the adjunction for $(\sigma_*,\sigma^!)$ and the second row of vertical morphisms comes from the pair $(\widetilde{i}_Z^*,\widetilde{i}_{Z*})$.

In summary, we have given isomorphisms of the map $\pi_! \phi_{g,1} \sigma_* \sigma^! \pi^! \to \pi_! \phi_{g,1} \pi^!$ with the adjunction morphism
\[ \pi_! \pi^! \to \pi_! \widetilde{i}_{Z*} \widetilde{i}_Z^* \pi^!,\]
which by Smooth Base Change \ref{eg-BaseChange} is isomorphic to $\pi_! \pi^!$ applied to the adjunction
\[ {\rm id} \to i_{Z*} i_Z^*.\]

Thus, using the fact that $\pi_! \pi^! \to {\rm id}$ is an isomorphism, we have given a natural isomorphism between the left-most morphism in the following exact triangles:
\[ \pi_! \phi_{g,1} \sigma_* \sigma^! \pi^! \to \pi_! \phi_{g,1} \pi^! \to \pi_! \phi_{g,1} j_* j^* \pi^! \xrightarrow[]{+1}\]
\[ {\rm id} \to i_{Z*} i_Z^* \to j_! j^* [1] \xrightarrow[]{+1},\]
which, by Lemma \ref{lem-localizationTriangles} below, gives a unique, natural isomorphism between the triangles.

The isomorphism between dual triangles can be shown similarly, or follows from applying duality.
\end{proof}

The following lemma was pointed out to the author by Martin Gallauer. The main idea is that localization triangles are rather special with regards to uniqueness of (morphisms between) cones.

\begin{lem} \label{lem-localizationTriangles} Let $i\colon Z \to X$ and $j\colon X \setminus Z \to X$ be closed and open embeddings of algebraic varieties, respectively. Let $A \to B\xrightarrow[]{\alpha} C \xrightarrow[]{\beta} A[1]$ be a distinguished triangle in $D^b({\rm MHM}(X))$ such that $A$ is supported on $Z$ and so that the morphism $A\to B$ fits into a commutative diagram
\[\begin{tikzcd} A \ar[r] \ar[d] & B \ar[d]\\ i_* i^! M \ar[r] & M \end{tikzcd}.\]

Then there is a \emph{unique} morphism $h\colon C \to j_* j^* M$ so that there is a morphism of triangles
\[\begin{tikzcd} A \ar[r] \ar[d] & B \ar[r,"\alpha"] \ar[d] & C \ar[r,"\beta"] \ar[d,"h"] & A[1] \ar[d]\\ i_* i^! M \ar[r] & M \ar[r] & j_* j^* M \ar[r] & i_*i^! M[1]\end{tikzcd}.\]

A similar statement holds for the dual localization triangle.
\end{lem}
\begin{proof} First of all, by Kashiwara's equivalence, we can write $A = i_* \overline{A}$ for some $\overline{A} \in D^b({\rm MHM}(Z))$, see \cite{SaitoMHM}*{Cor. 2.23}. By the theory of triangulated categories, there always exists at least one morphism $h\colon C \to j_*j^* M$ which gives a morphism of exact triangles.

If $h_1,h_2\colon C \to j_* j^* M$ are two such morphisms which give a morphism of exact triangles, then in particular, they satisfy $h_1 \circ \alpha = h_2 \circ \alpha$, and so $(h_1 - h_2)\circ \alpha = 0$. By the weak cokernel property of exact triangles, this gives rise to a morphism $\gamma\colon A[1] \to j_* j^* M$ so that $h_1 - h_2 = \gamma \circ \beta$. But then
\[ {\rm Hom}(A[1],j_*j^*M) \to {\rm Hom}(i_* \overline{A}[1],j_*j^* M) \cong {\rm Hom}(\overline{A}[1],i^! j_*j^* M) = 0,\]
using the adjunction $(i_*,i^!)$ and the fact that $i^! j_* = 0$. Thus, $\gamma = 0$ and so $h_1 = h_2$.
\end{proof}

\section{Radon Transform} \label{sect-Radon} This section consists of the proof of Theorem \ref{thm-Radon}. We begin by reducing the claim about the Grassmannian Radon transform to the ``affine" Radon transform. Let $Z'' = \{ (\pi(v),\lambda_1,\dots, \lambda_k) \in \P \times \bS(k,n) \mid \lambda_1(v) = \dots = \lambda_k(v) = 0\}$ with complement $C''$. We have the commutative diagrams, with each square being Cartesian:
\[ \begin{tikzcd} & Z' \ar[r, "\pi_2^{Z'}"] \ar[dl,swap,"\pi_1^{Z'}"] & \bhW \\ \P & Z'' \ar[d,"\pi_{Z''}"] \ar[u,swap,"r_{Z''}"] \ar[l,swap,"\pi_1^{Z''}"] \ar[r,"\pi_2^{Z''}"] & \bS(k,n) \ar[d,"\pi"] \ar[u,swap,"r"] \\ & Z \ar[ul,"\pi_1^Z"] \ar[r,"\pi_2^Z"] & \mathbf G \end{tikzcd}, \quad \begin{tikzcd} & C' \ar[r, "\pi_2^{C'}"] \ar[dl,swap,"\pi_1^{C'}"] & \bhW \\ \P & C'' \ar[d,"\pi_{C''}"] \ar[u,swap,"r_{C''}"] \ar[l,swap,"\pi_1^{C''}"] \ar[r,"\pi_2^{C''}"] & \bS(k,n) \ar[d,"\pi"] \ar[u,swap,"r"] \\ & C \ar[ul,"\pi_1^C"] \ar[r,"\pi_2^C"] & \mathbf G \end{tikzcd}.\]

The projections are induced by the ones in the following diagrams:
\[ \begin{tikzcd} & \P \times \bhW \ar[dl,swap,"\overline{p}"] \ar[dr,"\rho_W"] & \\ \P& & \bhW \end{tikzcd},\quad \begin{tikzcd} & \P \times \mathbf G \ar[dl,swap,"\pi_1"] \ar[dr,"\pi_2"] & \\ \P & & \mathbf G \end{tikzcd},\]
\[ \begin{tikzcd} & \P \times \bS(k,n) \ar[dl,swap,"\pi_1^S"] \ar[dr,"\pi_2^S"] & \\ \P & & \bS(k,n) \end{tikzcd}.\]

\begin{lem} \label{lem-compareGrassmannian} There is a natural isomorphism of triangles
\[ r^!(\pi_{2!}^{Z'} \pi_1^{Z'!} \to \rho_{W*} \overline{p}^! \to \pi_{2*}^{C'} \pi_1^{C'!}\xrightarrow[]{+1})\]
\[ \pi^! (\pi_{2!}^{Z} \pi_1^{Z!} \to \pi_{2*} \pi_1^! \to \pi_{2*}^{C} \pi_1^{C!}\xrightarrow[]{+1}),\]
and dually, there is a natural isomorphism of triangles
\[ r^*( \pi_{2!}^{C'} \pi_1^{C'*} \to \rho_{W!} \overline{p}^* \to \pi_{2*}^{Z'} \pi_1^{Z'*} \xrightarrow[]{+1} )\]
\[ \pi^*( \pi_{2!}^{C} \pi_1^{C*} \to  \pi_{2!} \pi_1^* \to \pi_{2*}^{Z} \pi_1^{Z*} \xrightarrow[]{+1}).\]
\end{lem}
\begin{proof} Note that $r$ is an open embedding and $\pi$ is smooth. These isomorphisms follow immediately from Base Change. 
\end{proof}

As noted in the proof of \cite{RadonFourier}*{Prop. 1}, it is important to view the bundle $\widetilde{\bV}$ as a $\mathbf G_m$-quotient
\[ \tau\colon \A^k_\zeta \times \bV^0 \to \widetilde{\bV},\quad (\zeta,v) \mapsto [\zeta_1 v,\zeta_2 v,\dots, \zeta_k v],\]
where the $\mathbf G_m$-action is $c(\zeta,x) = (c^{-1} \zeta, cx)$. The compositions with the natural maps $\widetilde{\jmath}\colon \widetilde{\bV} \to \bV^k$ and $\widetilde{\pi}\colon \widetilde{\bV} \to \P$ are 
\[ \widetilde{\jmath} \circ \tau\colon \A^k_\zeta \times \bV^0 \to \bV^k, (\zeta,v) \mapsto (\zeta_1 v, \zeta_2 v,\dots, \zeta_k v), \quad \widetilde{\pi} \circ \tau\colon \A^k_\zeta \times \bV^0 \to \P, (\zeta,v) \mapsto \pi(v).\]

The two key inputs are Theorem \ref{thm-MHMLocalization} and the following mixed Hodge module version of \cite{RadonFourier}*{Lem. 2}:
\begin{lem}[\cite{GequivMHM}*{Prop. 2.1}] \label{lem-smoothFF} Let $f\colon X \to Y$ be a smooth map between smooth complex algebraic varieties with connected fibers. Then the functors 
\[f^*, f^!\colon D^b({\rm MHM}(Y)) \to D^b({\rm MHM}(X))\]
are exact and fully faithful.
\end{lem}

Before proceeding with the proof, we give a commutative diagram which sets the notation we use in the proof below. This is our analogue of the diagram \cite{RadonFourier}*{Pg. 464}. The system of the notation is the following: $p$ denotes projection from the $\bhW$ term, $q$ is the projection away from the $\bV^k$ term, $\rho$ is the projection from the $\P$ term, and $(-)'$ is used for taking the product with $\bhW$.

\begin{equation} \label{eq-BigDiagram} \begin{tikzcd} \A^k_\zeta \times \bV^0 \ar[r,"\tau"] & \widetilde{\bV} \ar[r,"\widetilde{\iota}"] & \P \times \bV^k \\
\A^k_\zeta \times \bV^0 \times \bhW \ar[u,swap,"p_{12}"] \ar[r,"\tau' "] \ar[d,"p_\zeta'"] & \widetilde{\bV} \times \bhW \ar[r,"\widetilde{\iota}' "] \ar[u,swap,"\widetilde{p}"] \ar[d,"\widetilde{\pi}' "] \ar[dr,"\widetilde{\jmath}' "]& \P \times \bV^k \times \bhW \ar[u,swap,"q_{12}"] \ar[d,"\rho'"] \ar[dashed,dl,] \\ \bV^0 \times \bhW \ar[r,"\pi' "] & \P \times \bhW & \bV^k \times \bhW \end{tikzcd},\end{equation}
where the dashed arrow is denoted $p_{13}\colon \P \times \bV \times \bhW \to \P \times \bhW$.

We will prove the following, which by Lemma \ref{lem-compareGrassmannian} gives a proof of Theorem \ref{thm-Radon}:
\begin{prop} \label{prop-affineRadon} In the notation above, there is a natural isomorphism of triangles
\begin{equation} \label{eq-TriangleFL} {\rm FL} \eta_* \to {\rm FL} \widetilde{\jmath}_* \widetilde{\pi}^! \to {\rm FL} \widetilde{\jmath}_{\circ*}\widetilde{\pi}_\circ^! \xrightarrow[]{+1},\end{equation}
\begin{equation} \label{eq-TriangleAffineRadon} (\rho_{W!} \overline{p}^* \to \pi_{2*}^{Z'} \pi_1^{Z'*} \to \pi_{2!}^{C'} \pi_1^{C'*}[1] \xrightarrow[]{+1})(k(n+1))[k(n+1)] \end{equation}

Dually, there is an isomorphism of triangles
\[ {\rm FL} \widetilde{\jmath}_{\circ !} \widetilde{\pi}_{\circ}^* \to {\rm FL} \widetilde{\jmath}_! \widetilde{\pi}^* \to {\rm FL} \eta_* \xrightarrow[]{+1}\]
\[ (\pi_{2*}^{C'} \pi_1^{C'*}[-1] \to \pi_{2!}^{Z'} \pi_1^{Z'!} \to \rho_{W!} \overline{p}^* \xrightarrow[]{+1})[-k(n+1)].\]
\end{prop}

\begin{proof}[Proof of Prop. \ref{prop-affineRadon}] 
Let $E = \P \times \{0\} \subseteq \widetilde{\bV}$ be the zero section of the bundle with closed embedding $i_E\colon E \to \widetilde{\bV}$. The complement is $j_\circ \colon \widetilde{\bV}^\circ \to \widetilde{\bV}$. We have the localization triangle
\begin{equation}\label{eq-FLTriangle} i_{E*} i_E^! \widetilde{\pi}^! \to \widetilde{\pi}^! \to j_{\circ *} j_\circ^* \widetilde{\pi}^! \xrightarrow[]{+1} \end{equation}

First, note that applying ${\rm FL}\widetilde{\jmath}_*$ to this triangle gives Triangle \ref{eq-TriangleFL} from the Proposition statement. Indeed, note that $\widetilde{\jmath}\circ i_E\colon \P \times \{0\} \to \bV^k$ is the projection onto $0 \in \bV^k$, which we denote $\eta\colon \P \to \bV^k$, and we have $\widetilde{\pi} \circ i_E = {\rm id}_{\P}$.

We write
\[ {\rm FL}\widetilde{\jmath}_* = q_! \phi_g p^! \widetilde{\jmath}_*[-k(n+1)],\]
where $\dim(\bV^k) =k(n+1)$. The Cartesian diagram
\[ \begin{tikzcd} \widetilde{\bV} \times \bhW \ar[d,"\widetilde{p}"] \ar[r,"\widetilde{\iota}'"] & \P \times \bV^k \times \bhW \ar[r,"\rho'"] \ar[d] & \bV^k \times \bhW \ar[d,"p"] \\ \widetilde{\bV} \ar[r,"\widetilde{\iota}"] & \P \times \bV^k \ar[r,"\rho"] & \bV^k \end{tikzcd},\]
gives by Base Change a natural isomorphism $p^! \widetilde{\jmath}_* \cong \widetilde{\jmath}'_* \widetilde{p}^!$. As $\widetilde{\jmath}'$ is proper, if we let $\widetilde{g} = (\widetilde{\jmath}')^*(g)$, then this gives a natural isomorphism
\[ {\rm FL} \widetilde{\jmath}_* \cong q_! \widetilde{\jmath}'_* \phi_{\widetilde{g}} \widetilde{p}^![-k(n+1)].\]

Note that $q \circ \widetilde{\jmath}' = \rho_W \circ \widetilde{\pi}'$. So the right hand side is naturally isomorphic to
\[ \rho_{W!} \widetilde{\pi}'_! \phi_{\widetilde{g}} \widetilde{p}^![-k(n+1)].\]

On the other hand, Triangle \ref{eq-TriangleAffineRadon} is, by definition $\rho_{W!}$ applied to the triangle
\[\overline{p}^* \to i_{Z'*} i_{Z'}^* \overline{p}^* \to j_{C'!} j_{C'}^* \overline{p}^*[1] \xrightarrow[]{+1}, \]
which it will be convenient to rewrite (using the fact that $\overline{p}$ is smooth of relative dimension $k(n+1)$) as 
\begin{equation} \label{eq-TriangleAffineRadon2} (\overline{p}^! \to i_{Z'*} i_{Z'}^* \overline{p}^! \to j_{C'!} j_{C'}^* \overline{p}^![1] \xrightarrow[]{+1})(-k(n+1))[-2k(n+1)].\end{equation}

So we need to give an isomorphism of Triangle \ref{eq-TriangleAffineRadon2} with $\widetilde{\pi}'_! \phi_{\widetilde{g}} \widetilde{p}^! [-k(n+1)]$ applied to Triangle \ref{eq-FLTriangle}, and then to apply $\rho_{W!}$ to that isomorphism. We will do this by showing that $(\pi')^!$ applied to those triangles yields isomorphic triangles, which completes the proof by Lemma \ref{lem-smoothFF}.

We have the diagram
\[ \begin{tikzcd} \cG \ar[r,"i_{\cG}"] \ar[d] & \bV^0 \times \bhW \ar[d,"\pi'"] & \cU \ar[l,swap,"j_{\cU}"] \ar[d] \\ Z' \ar[r,"i_{Z'}"] & \P \times \bhW & C' \ar[l,swap,"j_{C'}"] \end{tikzcd},\]
and so $(\pi')^!$ applied to Triangle \ref{eq-TriangleAffineRadon2} gives
\[  ((\pi')^! \overline{p}^! \to i_{\cG*} i_{\cG}^* (\pi')^! \overline{p}^! \to j_{\cU!} j_{\cU}^* (\pi')^! \overline{p}^![1] \xrightarrow[]{+1})(-k(n+1))[-2k(n+1)],\]
which, because $\overline{p} \circ \pi' = \pi \circ p_0$ with $p_0\colon \bV^0 \times \bhW \to \bV^0$ being the projection, is isomorphic to
\begin{equation} \label{eq-TriangleAffineRadon3} ( p_0^! \pi^! \to i_{\cG*} i_{\cG}^* p_0^! \pi^!\to j_{\cU!} j_{\cU}^* p_0^! \pi^! [1] \xrightarrow[]{+1})(-k(n+1))[-2k(n+1)].\end{equation}

Note that, by definition of $Z'$, the subvariety $\cG \subseteq \bV^0 \times \bhW$ is defined by $g^{(1)},\dots, g^{(k)}$, where if $z_0,\dots, z_n$ are coordinates on $\bV$ (hence, on $\bV^0$) and $w_0^{(i)},\dots, w_n^{(i)}$ are dual coordinates on the $i$th copy of $\widehat{\bV}$ (where $\bhW = \widehat{\bV}^k$), then
\[ g^{(i)} = \sum_{j=0}^n z_j w_j^{(i)}.\]

On the other hand, Base Change tells us that
\[ (\pi')^! (\widetilde{\pi}')_! \cong p_{\zeta!}' \circ (\tau')^!,\]
and so we have an isomorphism
\[ (\pi')^! \widetilde{\pi}'_! \phi_{\widetilde{g}} \widetilde{p}^![-k(n+1)] \cong  p_{\zeta!}' (\tau')^! \phi_{\widetilde{g}} \widetilde{p}^! [-k(n+1)].\]

As $\tau'$ is smooth, if we write $g_\zeta = (\tau')^*(\widetilde{g})$, we get by compatibility of vanishing cycles with smooth morphisms that this is isomorphic to
\[ p_{\zeta!}' \phi_{g_\zeta} (\tau')^! \widetilde{p}^! [-k(n+1)].\]

By definition of $\tau'$, we see that $(\tau')^*(\widetilde{g}) = \sum_{i=1}^k \zeta_i g^{(i)}$. Moreover, we have $\widetilde{p}\circ \tau' = \tau \circ p_{12}$, and so we are interested in the functor
\[ p_{\zeta!}' \phi_{g_\zeta} p_{12}^! \tau^![-k(n+1)].\]

Recall that we are applying this functor to Triangle \ref{eq-FLTriangle}. We have the diagram
\[ \begin{tikzcd} \{0\} \times \bV^0 \times \bhW \ar[r,"i_0'"] \ar[d] & \A^k_\zeta \times \bV^0 \times \bhW \ar[d,"p_{12}"] & \cU_0 \times \bhW \ar[l,swap,"j_{\cU_0}'"] \ar[d]  \\ \{0\} \times \bV^0  \ar[r,"i_0"] \ar[d] & \A^k_\zeta \times \bV^0 \ar[d,"\tau"] & \cU_0\ar[l,swap,"j_{\cU_0}"] \ar[d] \\ E \ar[r,"i_E"]&\widetilde{\bV} & \widetilde{\bV}^\circ \ar[l,swap,"j_\circ"]\end{tikzcd}.\]

Hence, $p_{12}^! \tau^!$ applied to Triangle \ref{eq-FLTriangle} is isomorphic to
\[ i_{0*}' (i_0')^! p_{12}^!\tau^!\widetilde{\pi}^! \to p_{12}^!\tau^!\widetilde{\pi}^! \to j_{\cU_0*}' (j_{\cU_0}')^! p_{12}^!\tau^!\widetilde{\pi}^! \xrightarrow[]{+1}.\]

The composition $\widetilde{\pi} \circ \tau \circ p_{12}$, which is the projection onto $\P$, can be factored as $\pi \circ p_0 \circ p'_{\zeta}$. In other words, we can rewrite this triangle as 
\begin{equation} \label{eq-FLTriangle2} i_{0*}' (i_0')^! (p'_\zeta)^! p_0^! \pi^!\to (p'_\zeta)^! p_0^! \pi^! \to j_{\cU_0*} j_{\cU_0}^! (p'_\zeta)^! p_0^! \pi^! \xrightarrow[]{+1}.\end{equation}

Recalling that $g_\zeta = \sum_{i=1}^k \zeta_i g^{(i)}$, we can now apply Theorem \ref{thm-MHMLocalization}, which says that the functor $p_{\zeta!}' \phi_{g_\zeta}[-k(n+1)]$ applied to Triangle \ref{eq-FLTriangle2} is naturally isomorphic to
\[ (p_0^! \pi^! \to i_{\cG*} i_\cG^* p_0^! \pi^! \to j_{\cU!} j_{\cU}^* p_0^! \pi^! [1] \xrightarrow[]{+1})[-k(n+1)].\]

This is Triangle \ref{eq-TriangleAffineRadon3}, but with a Tate twist by $(k(n+1))$ and a shift by $[k(n+1)]$, as desired.

The isomorphism of the dual triangles can be proven similarly, or it follows immediately by duality, using $\mathbf D \circ {\rm FL} = \overline{\rm FL} \circ \mathbf D = ({\rm FL} \circ \mathbf D)(-k(n+1))$.
\end{proof}

\section{Application to GKZ Systems} In this section, we use the comparison (Proposition \ref{prop-affineRadon} with $k=1$) between the Fourier-Laplace transform and the (affine) Radon transform to show that two naturally defined mixed Hodge module structures (one in \cite{ReicheltGKZ}*{Thm. 3.5} and the other in \cite{GKZ}*{Cor. 4.14}) on certain GKZ systems are isomorphic up to a Tate twist. This answers affirmatively the question in \cite{GKZ}*{Rmk. 4.15}.

We begin with the definition of GKZ systems, for more details, one should consult \cite{OriginalGKZ, SWGKZ,ReicheltGKZ, GKZ}. We use the notation of \cite{ReicheltGKZ}. Consider an integer $(d\times n)$-matrix $A$ with columns $a_1,\dots, a_n \in \Z^d$. Let $\mathbf L$ be the $\Z$-module of relations between the columns of $A$. We consider the cyclic $\cD_{\A_\lambda^n}$-module
\[ \cM_A^\beta = \cD_{\A_\lambda^n} /((\Box_\ell)_{\ell \in \mathbf L} + (E_k - \beta_k)),\]
where, for $\ell \in \mathbf L$,
\[ \Box_{\ell} = \prod_{i, \ell_i < 0} \de_{\lambda_i}^{-\ell_i} - \prod_{i, \ell_i >0} \de_{\lambda_i}^{\ell_i},\]
and $E_k = \sum_{i=1}^n a_{ki} \lambda_i \de_{\lambda_i}$ is the weighted Euler operator associated to the $k$th row of $A$. We let $\widehat{\cM}_A^\beta$ be the inverse Fourier transform of $\cM_A^\beta$ as a $\cD$-module.

In \cite{SWGKZ}, the authors define, for a fixed matrix $A$, a set of parameters $\beta$ which should be avoided in order to have a well behaved GKZ system. This is the set of \emph{strongly resonant parameters}, denoted ${\rm sRes}(A)$, whose definition we do not review here. The utility of this set is explained by Theorem \ref{thm-sRes} below.

Consider a torus $T = {\rm Spec}(\C[t_1^{\pm 1},\dots, t_d^{\pm 1}])$. The matrix $A$ defines for us a morphism
\[ h\colon T \to \A^n_\mu = \bV, \quad (t_1,\dots, t_d) \mapsto (t^{a_1},\dots, t^{a_n}),\]
where $t^{a_j} = \prod_{i=1}^d t_i^{a_{ij}}$ and whose image has closure equal to $Y' = {\rm Spec}(\C[\bN A])$, where $\bN A$ is the semigroup ring associated to the semigroup $A$. We say that the matrix $A$ is \emph{pointed} if $\bN A \cap (-\bN A) = 0$.

We have the $\cD_T$-module $\cO_T^\beta$ defined as the quotient
\[ \cD_T/ \cD_T(t_1 \de_{t_1}+ \beta_1,\dots, t_d \de_{t_d} + \beta_d).\]

The main starting point is the following theorem which gives the structure of the module $\widehat{\cM}_A^\beta$:
\begin{thm}\label{thm-sRes} $($\cite{SWGKZ}*{Thm. 3.6, Cor. 3.7}$)$ Let $A$ be a pointed $(d\times n)$-integer matrix whose columns span $\Z^d$. Then for the map $h\colon T \to \bV$ defined above, the following are equivalent: \begin{enumerate}
\item $\beta \notin {\rm sRes}(A)$.
\item $\widehat{\cM}_A^\beta \cong h_+ \cO_T^\beta$.
\end{enumerate}
\end{thm}

For $\beta \in \Z^d$, we have a canonical isomorphism $\cO_T^\beta \cong \cO_T$. Hence, if $\beta \in \Z^d \setminus {\rm sRes}(A)$, then there is an isomorphism
\[ \widehat{\cM}_A^\beta \cong h_+ \cO_T.\]

So we obtain a mixed Hodge module structure on $\hat{\cM}_A^\beta$ by
\[ {}^H \widehat{\cM}_A^\beta = h_* \Q_T^H[d].\]

If moreover $A$ is homogenous, i.e., $(1,\dots, 1)^{\rm t} \in \Z(A^{\rm t})$, then the module ${}^H \widehat{\cM}_A^\beta$ is monodromic by \cite{ReicheltGKZ}*{Lem. 1.13}. Hence, as in \cite{GKZ}, we can apply the Fourier-Laplace transform for monodromic mixed Hodge modules and obtain a mixed Hodge module structure on $\cM_A^\beta$:
\[ {}^H \cM_A^\beta \cong {\rm FL}(h_* \Q_T^H[d]).\]

We have the following diagram
\[ \begin{tikzcd} T \ar[dd,"{\rm pr}"] \ar[rd,"h_0"] \ar[r,"h"] & \bV\\ & \bV\setminus \{0\} \ar[u,"j_0"] \ar[d,"\pi"] \\ \overline{T} \ar[r,"\overline{h}"] & \P(\bV) \end{tikzcd},\]
where ${\rm pr}\colon T \to \overline{T} = (\C^*)^{d-1}$ is the projection onto the last $d-1$ terms, $\overline{h}$ is the projectivization of $h$, and by definition, $h$ factors through $\bV \setminus \{0\}$.

Then, as in the proof of \cite{GKZ}*{Cor. 4.14}, we have
\[ h_* \Q_T^H[d] \cong h_*({\rm pr}^* \Q_{\overline{T}}^H)[d] \cong h_*({\rm pr}^! \Q_{\overline{T}}^H(-1)[-2])[d] \cong j_{0*} h_{0*}({\rm pr}^! \Q_{\overline{T}}^H)(-1)[d-2] \]
\[\cong j_{0*} \pi^!(\overline{h}_* \Q_{\overline{T}}^H)(-1)[d-2].\]

By applying ${\rm FL}$, we get an isomorphism
\[ {}^H \cM_A^\beta \cong {\rm FL}(j_{0*} \pi^!(\overline{h}_*\Q_{\overline{T}}^H))(-1)[d-2],\]
and the right hand side is in precisely the form to compare with the affine Radon transformation. By Proposition \ref{prop-affineRadon} with $k = 1$ and $n$ in place of $n+1$, the right hand side is isomorphic to
\[ \pi_{2!}^{C'} \pi_1^{C'*}(\overline{h}_*\Q_{\overline{T}}^H)(n-1)[d+n-1].\]

As ${}^H \cM_A^\beta$ is a single mixed Hodge module in degree 0, this gives an isomorphism
\[ {}^H \cM_A^\beta \cong \cH^{d+n-1} \pi_{2!}^{C'} \pi_1^{C'*}(\overline{h}_*\Q_{\overline{T}}^H)(n-1),\]
where the right hand side is the definition of the mixed Hodge module structure from \cite{ReicheltGKZ}*{Defi. 3.2}, up to a Tate twist by $(n-1) = \dim \P(V)$.

\end{document}